\documentclass{lms}

\usepackage{amssymb,amsmath,amsxtra}
\usepackage[all]{xy}
\usepackage{comment}

\numberwithin{equation}{section}

\newtheorem{lemma}[equation]{Lemma}     
\newtheorem{corollary}[equation]{Corollary}
\newtheorem{proposition}[equation]{Proposition}

\newnumbered{assertion}[equation]{Assertion}    
\newnumbered{conjecture}[equation]{Conjecture}  
\newnumbered{definition}[equation]{Definition}
\newnumbered{hypothesis}[equation]{Hypothesis}
\newnumbered{remark}[equation]{Remark}
\newnumbered{note}[equation]{Note}
\newnumbered{observation}[equation]{Observation}
\newnumbered{problem}[equation]{Problem}
\newnumbered{question}[equation]{Question}
\newnumbered{algorithm}[equation]{Algorithm}
\newnumbered{example}[equation]{Example}
\newunnumbered{theorema}{Theorem A}
\newunnumbered{theoremb}{Theorem B}
\newunnumbered{notation}{Notation} 

\newcommand{\psmod}[1]{~(\textup{\text{mod}}~{#1})}

\DeclareMathOperator{\Frob}{Frob}
\DeclareMathOperator{\GL}{GL}
\DeclareMathOperator{\opchar}{char}

\DeclareMathOperator{\lm}{lm}
\DeclareMathOperator{\ord}{ord}
\DeclareMathOperator{\opspan}{span}
\DeclareMathOperator{\supp}{supp}
\DeclareMathOperator{\Tr}{Tr}
\DeclareMathOperator{\vol}{vol}
\DeclareMathOperator{\Vol}{Vol}

\newcommand{\A}{\mathbb A}
\newcommand{\C}{\mathbb C}
\newcommand{\F}{\mathbb F}
\newcommand{\G}{\mathbb G}

\newcommand{\PP}{\mathbb P}
\newcommand{\NN}{\mathbb N}
\newcommand{\Q}{\mathbb Q}
\newcommand{\R}{\mathbb R}

\newcommand{\Z}{\mathbb Z}

\newcommand{\frakp}{\mathfrak p}

\newcommand{\omegadot}{{}^{\bullet}}

\newcommand{\la}{\langle}
\newcommand{\ra}{\rangle}

\newcommand{\fbar}{\overline{f}}
\newcommand{\xbar}{\overline{x}}
\newcommand{\ybar}{\overline{y}}
\newcommand{\wbar}{\overline{w}}

\newcommand{\Vbar}{\overline{V}}

\title[Computing zeta functions]{Computing zeta functions of nondegenerate hypersurfaces \\ with few monomials}
\author{Steven Sperber and John Voight}
\classno{11Y16 (primary), 11M38, 14D10, 14F30 (secondary)}

\extraline{The second author was partially supported by the National Security Agency under Grant Number H98230-09-1-0037.}

\begin{document}
\maketitle

\begin{abstract}
Using the cohomology theory of Dwork, as developed by Adolphson and Sperber, we exhibit a deterministic algorithm to compute the zeta function of a nondegenerate hypersurface defined over a finite field.  This algorithm is particularly well-suited to work with polynomials in small characteristic that have few monomials (relative to their dimension).  Our method covers toric, affine, and projective hypersurfaces and also can be used to compute the $L$-function of an exponential sum.
\end{abstract}

Let $p$ be prime and let $\F_q$ be a finite field with $q=p^a$ elements.  Let $\Vbar$ be a variety defined over $\F_q$, described by the vanishing of a finite set of polynomial equations with coefficients in $\F_q$.  We encode the number of points $\#\Vbar(\F_{q^r})$ on $\Vbar$ over the extensions $\F_{q^r}$ of $\F_q$ in an exponential generating series, called the \emph{zeta function} of $\Vbar$:
\[ Z(\Vbar,T)=\exp\left(\sum_{r=1}^{\infty} \#\Vbar(\F_{q^r})\frac{T^r}{r}\right) \in 1+T\Z[[T]]. \]
The zeta function $Z(\Vbar,T)$ is a rational function in $T$, a fact first proved using $p$-adic methods by Dwork \cite{Dwork1,Dwork2}.  The algorithmic problem of computing $Z(\Vbar,T)$ efficiently is of significant foundational interest, owing to many practical and theoretical applications (see e.g.\ Wan \cite{Wan1} for a discussion).  

From a modern point of view, we consider $Z(\Vbar,T)$ cohomologically: we build a $p$-adic cohomology theory that functorially associates to $\Vbar$ certain vector spaces $H^i$ over a $p$-adic field $K$, each equipped with a (semi-)linear operator $\Frob_i$, such that $Z(\Vbar,T)$ is given by an alternating product of the characteristic polynomials of $\Frob_i$ acting on the spaces $H^i$.  The theory of $\ell$-adic \'etale cohomology, for example, was used by Deligne to show that $Z(\Vbar,T)$ satisfies a Riemann hypothesis when $\Vbar$ is smooth and projective.  Parallel developments have followed in the $p$-adic (de Rham) framework, including the theories of Monsky-Washnitzer, crystalline, and rigid cohomology (see Kedlaya \cite{Kedlayaoverview} for an introduction).  In this paper, for a toric hypersurface $\Vbar$ defined by a (nondegenerate) Laurent polynomial $\fbar$ in $n$ variables over $\F_q$, we employ the cohomology theory of Dwork, working with a space $H^{n+1}(\Omega\omegadot)$ obtained as the quotient of a $p$-adic power series ring over $K$ in $n + 1$ variables by the subspace generated by the images of $n + 1$ differential operators.

Efforts to make these cohomology theories computationally effective have been extensive.  Schoof's algorithm for counting points on an elliptic curve \cite{Schoof} (generalized by Edixhoven and his coauthors \cite{Edixhovenproject} to compute coefficients of modular forms) can be viewed in this light, using the theory of mod $\ell$ \'etale cohomology.  A number of results on the $p$-adic side have also emerged in recent years.  In early work, Wan \cite{Wan2} and Lauder and Wan \cite{LauderWan} demonstrated that the $p$-adic methods of Dwork can be used to efficiently compute zeta functions in small (fixed) characteristic.  Lauder and Wan use the Dwork trace formula and calculate the trace of Frobenius acting on a $p$-adic Banach space, following the original method of Dwork and working on the ``chain level''.  In this paper, we instead work with the extension of Dwork's theory due to Adolphson and Sperber \cite{ASExpsums}; this point of view was also pursued computationally by Lauder and Wan in the special case of Artin-Schreier curves \cite{LauderWanArtinSchreier,LauderWanArtinSchreierII}.  Under the hypothesis that the Laurent polynomial $\fbar$ is \emph{nondegenerate} (see below for the precise definition), the zeta function can be recovered from the action of Frobenius on a certain single cohomology space $H^{n+1}(\Omega)$.  This method works with exponential sums and so extends naturally to the case of toric, affine, or projective hypersurfaces \cite{ASother}.  (It suffices to consider the case of hypersurfaces to compute the zeta function of any variety defined over a finite field using inclusion-exclusion or the Cayley trick.)

The method of Dwork takes into account the terms that actually occur in the Laurent polynomial $\fbar$; these methods are especially well-suited when the monomial support of $\fbar$ is small, so that certain combinatorial aspects are simple.  This condition that $\fbar$ have few monomials in its support, in which case we say (loosely) that $\fbar$ is \emph{fewnomial} (a term coined by Kouchnirenko \cite{Kouchfew}), is a natural one to consider.  For example, many explicit families of hypersurfaces of interest, including the well-studied (projective) Dwork family $x_0^{n+1}+\dots+x_n^{n+1}+\lambda x_0x_1\cdots x_n=0$ of Calabi-Yau hypersurfaces \cite{Dwork3} (as well as more general monomial deformations of Fermat hypersurfaces \cite{DworkPAA}) can be written with few monomials.  In cryptographic applications, the condition of fewnomialness also often arises.  Finally, the running time of algorithms on fewnomial input are interesting to study from the point of view of complexity theory: see, for example, work of Bates, Bihan, and Sottile \cite{BBS}.  

To introduce our result precisely, we now set some notation.  Let $\Vbar$ be a toric hypersurface, the closed subset of $\G_m^n$ defined by the vanishing of a Laurent polynomial
\[ \fbar=\sum_{\nu \in \Z^n} \overline{a}_\nu x^\nu \in \F_q[x^{\pm}]=\F_q[x_1^{\pm},\dots,x_n^{\pm}]. \] 
We use multi-index notation, so $x^\nu=x_1^{\nu_1}\cdots x_n^{\nu_n}$.  We sometimes write $Z(\fbar,T)=Z(\Vbar,T)$.  Let $\Delta=\Delta(\fbar)$ be the \emph{Newton polytope} of $\fbar$, the convex hull of its \emph{support} 
\[ \supp(\fbar)=\{\nu \in \Z^n:\overline{a}_\nu \neq 0\} \]  
in $\R^n$.  For simplicity, we assume throughout that $\dim(\Delta)=n$.  For a face $\tau \subseteq \Delta$, let $\fbar|_\tau=\sum_{\nu \in \tau} \overline{a}_\nu x^\nu$.  Then we say $\fbar$ is \emph{($\Delta$-)nondegenerate} if for all faces $\tau \subseteq \Delta$ (including $\Delta$ itself), the system of equations
\[ \fbar|_\tau = x_1\frac{\partial \fbar|_\tau}{\partial x_1} = \cdots = x_n\frac{\partial \fbar|_\tau}{\partial x_n} = 0 \]
has no solution in $\overline{\F}_q^{\times n}$, where $\overline{\F}_q$ is an algebraic closure of $\F_q$.  The set of $\Delta$-nondegenerate polynomials with respect to a polytope $\Delta$ forms an open subset in the affine space parameterizing their coefficients $(\overline{a}_\nu)_{\nu \in \Delta \cap \Z^n}$: under mild hypothesis, such as when $\Delta$ contains a unimodular simplex, then this subset is Zariski dense.  (See Batyrev and Cox \cite{BatyrevCox} as a reference for this notion as well as the work of Castryck and the second author \cite{CV} for a detailed analysis of nondegenerate curves.)  We distinguish here between $\Delta(\fbar)$ and $\Delta_\infty(\fbar)$ which is the convex closure of $\Delta(\fbar) \cup \{0\}$: for the Laurent polynomial $w\fbar$ in $n+1$ variables, $\fbar$ is $\Delta$-nondegenerate if and only if $w\fbar$ is nondegenerate with respect to $\Delta_\infty(\fbar)$ in the sense of Kouchnirenko \cite{Kouchnirenko}, Adophson and Sperber \cite{ASExpsums}, and others.  Nondegenerate hypersurfaces are an attractive class to consider because many of their geometric properties can be deduced from the combinatorics of their Newton polytopes.

Let $s=\#\supp(\fbar)$ and let $U$ be the $(n+1) \times s$-matrix with entries in $\Z$ whose columns are the vectors $(1,\nu) \in \Z^{n+1}$ for $\nu \in \supp(\fbar)$.  Let $\rho$ be the rank of $U$ modulo $p$.  Let $v=\Vol(\Delta)=n!\vol(\Delta)$ be the normalized volume of $\Delta$, so that a unit hypercube $[0,1]^n$ has normalized volume $n!$ and the unit simplex $\sigma=\{(a_1,\dots,a_n) \in \R_{\geq 0}^n : \sum_i a_i \leq 1\}$ has normalized volume $1$.

We say that $\Delta$ is \emph{confined} if $\Delta$ is contained in an orthotope (box) with side lengths $b_1,\dots,b_n$ with $b_1 \cdots b_n \leq n^n v$.  We say that $\fbar$ is \emph{confined} if $\Delta(\fbar)$ is confined.  A slight extension of a theorem of Lagarias and Ziegler \cite{LagariasZiegler} shows that every polytope $\Delta$ is $\GL_n(\Z)$-equivalent to a confined polytope; this existence can also be made effective.  (See section 3 for more detail.)  In other words, for each Laurent polynomial $\fbar$ there is a computable monomial change of basis of $\F_q[x^{\pm}]$, giving rise to an equality of zeta functions, under which $\fbar$ is confined.  (In the theorem below, at the expense of introducing a factor of $\log \delta$, where $\delta=\delta(S)=\max_{\nu \in S} |\nu|$ where $|\nu|=\max_i |\nu_i|$, one can remove the assumption that $\Delta$ is confined.)

For functions $f,g:\Z_{\geq 0}^m \to \R_{\geq 0}$, we say that $f=O(g)$ if there exists $c \in \R_{>0}$ and $N \in \Z_{\geq 0}$ such that for every $x=(x_1,\dots,x_m) \in \Z_{\geq N}^m$ we have $g(x) \leq c f(x)$.  (The reader is warned that not all properties familiar to big-Oh notation for functions of one variable extend to the multivariable case; see Howell \cite{Howell}.  In fact, our analysis also holds with Howell's more restrictive definition, but we will not pursue this further here.)  We further use the ``soft-Oh'' notation, where $f=\widetilde{O}(g)$ if $f=O(g \log^k g)$ for some $k \geq 1$.

Our main result is as follows.

\begin{theorema}
Let $n \in \Z_{\geq 1}$.  Then there exists an explicit algorithm that, on input a nondegenerate Laurent polynomial $\fbar \in \F_q[x_1^{\pm},\dots,x_n^{\pm}]$ with $p \geq 3$ and an integer $N \geq 1$, computes as output $Z(\fbar,T)$ modulo $p^N$.  If further $\fbar$ is confined, then this algorithm uses
\[ \widetilde{O}\bigl(s^{\lceil n/2 \rceil} + p N^3 \log q + p^{s-\rho}(6N+n)^s (v^4 N \log^2 q)\bigr) \] 
bit operations.
\end{theorema}

To recover the zeta function (as an element of $\Q(T)$), if we fix both the dimension and the number $s$ of monomials, we have the following result.

\begin{theoremb}
Let $n \in \Z_{\geq 1}$ and $s \in \Z_{>n}$.  Then there exists an explicit algorithm that, on input a confined, nondegenerate Laurent polynomial $\fbar \in \F_q[x_1^{\pm},\dots,x_n^{\pm}]$ with $p \geq 5$ and $s=\#\supp(\fbar)$, computes as output $Z(\fbar,T)$ using
\[ \widetilde{O}(p^{\min(1,s-\rho)} v^{s+5} \log^{s+3} q)  \] 
bit operations.
\end{theoremb}

According to a theorem of Adolphson and Sperber \cite{ASExpsums}, under the hypothesis that $\fbar$ is nondegenerate, $Z(\fbar,qT)^{(-1)^n}$ is a polynomial of degree $v$ times $Z(\G_m^n,T)^{(-1)^n}$.  Therefore, in the context of Theorem B, if $p=O(v \log q)$ is small (or fixed), then our algorithm runs in polynomial time in the (dense) output size, which is the best one could hope for (aside from minimizing the degree of this polynomial).  (The fewnomial input size, on the other hand, is $O(s \log v \log q)$ for $\fbar$ confined and $n$ fixed.)

In our theorem, we require the dimension $n$ to be fixed for several reasons.  First, we employ well-known algorithms for lattice polytopes which have only been analyzed assuming that the dimension is fixed.  (They further assume that arithmetic operations in $\Z$ take time $O(1)$, which is nearly valid in the usual bit-complexity model if $\Delta$ is confined and $n$ is fixed; for a discussion of this point, see Section 3.)  Second, it is often quite natural from a geometric point of view to consider the dimension to be fixed; one often considers families of hypersurfaces of a fixed dimension, for example.  Finally, allowing fewnomial input and output and varying dimension, the problem of computing $Z(\fbar,T)$ is harder than the NP-complete problem 3-\textsf{SAT} (indeed, for the latter one only wishes to know if $\#X(\F_2) > 0$ for an affine hypersurface $X$ of degree $2$).  For these reasons, we restrict our analysis (continuing below) to fixed dimension.

Our method follows in the same vein as other recently introduced $p$-adic cohomological techniques.  The methods of Lauder and Wan \cite{LauderWan} mentioned above compute $Z(\fbar,T)$ for a polynomial $\fbar$ of total degree $d$ using $\widetilde{O}(p^{2n+4}d^{3n^2+9n}\log^{3n+7} q)=\widetilde{O}((p v \log q)^{3n+9})$ bit operations.  The dense input size of $\fbar$ is $O((d+1)^n \log q)$; consequently if the prime $p$ (and dimension $n$) are fixed then their algorithm runs in polynomial time in the dense input size.  Their method, although apparently not practical, is completely general and does not require any hypothesis on $\fbar$.  Our method can be analyzed on dense input (see Section 5) as well, running in time $\widetilde{O}(p^{2n}v^{2n+4}\log^{2n+2})$ with no condition on the number of monomials in the support of $\fbar$.  

In a different direction, Kedlaya \cite{Kedlaya} (see also the presentation by Edixhoven \cite{Edixhoven}) used Monsky-Washnitzer cohomology to compute the zeta function of a hyperelliptic curve of genus $g$ over $\F_q$ in time $\widetilde{O}(pg^4 \log^3 q)$.  (Note here that the dense input size is $O(g\log q)$.) This idea has been taken up by several others: see, for example, work of Abbott, Kedlaya, and Roe \cite{akr}, who compute the zeta function of a projective hypersurface by working in the complement of the hypersurface and using Mumford reduction.  (Indeed, Kedlaya has suggested that there should be a natural extension \cite{Kedlayapreprint} of his  ideas to the realm of toric hypersurfaces.)  Our method also mirrors the algorithm of Castryck, Denef, and Vercauteren \cite{cdv}, who tackle the case of nondegenerate curves.  Their method has good asymptotic behavior but to be practical needs an optimized implementation \cite[\S 1.2.4]{castryckthesis}. However, rather than following this vein and working with Monsky-Washnitzer ($p$-adic de Rham) cohomology, we employ the cohomology theory of Dwork, which has a more combinatorial flavor.

In a yet further direction, Lauder has used Dwork's theory of $p$-adic differential equations to compute zeta functions using deformation \cite{LauderDeformation} and recursion \cite{LauderRecursion}.  The Frobenius, acting on the members of a one-parameter family, satisfies a differential equation coming from the Gauss-Manin connection.  Lauder uses this equation to solve for the action given an initial condition arising from the action on the cohomology of a simple variety which one can compute directly.  Our method fits into this framework as it provides natural base varieties to deform from: indeed, the idea of deformation in the context of nondegenerate curves has been pursued by Tuitman \cite{Tuitman}.  The methods of Lauder show that one can compute $Z(\Vbar,T)$ for a smooth projective hypersurface $\Vbar \subseteq \PP^n$ of degree $d$ with $p \nmid d$ and nonvanishing diagonal terms in time $p^2(d^n \log q)^{O(1)}$.  The deformation method has also been pursued fruitfully by Gerkmann \cite{Gerkmann} and others in different contexts.

Adapting an idea of Chudnovsky and Chudnovsky and Bostan, Gaudry, and Schost \cite{BGS} for accelerated reduction, Harvey \cite{Harvey} has improved Kedlaya's method for hyperelliptic curves, with a runtime of $p^{1/2}(g \log q)^{O(1)}$.  This approach appears to extend to higher dimensions as well, extending the method of Abbott, Kedlaya, and Roe \cite{Harveypreprint}: his method appears to give a runtime of $p^{1/2} (d^n \log q)^{n+O(1)}$ under a smoothness hypothesis analogous to the condition of nondegeneracy (but somewhat weaker).  It would be interesting to see how his ideas for lowering the exponent on $p$ might apply in our situation.

This paper is organized as follows.  In section 1, we introduce the cohomology theory of Dwork and give an overview of our method.  In section 2, we discuss each step of the algorithm in turn: computing the splitting function and the Jacobian ring, the computation of Frobenius (where the condition of sparsity enters), and the reduction theory in cohomology.  In section 3, we give some algorithms for computing with polytopes.  We then discuss running time and precision estimates for the complete algorithm in section 4.  Finally, in section 5 we discuss the case $p=2$ and consider some other possible modifications.  We conclude in section 6 with some examples.

\section{Overview}

In this section, we give an overview of our algorithm.  Our introduction will be concise; for a more complete treatment of the theory of Dwork \cite{Dwork2}, see Koblitz \cite{Koblitzbook}, Lauder and Wan \cite{LauderWan}, and Monsky \cite{Monsky}.

In this section, we assume $p>2$; see section 5 for a discussion of the case $p=2$.

\subsection*{Exponential sums}

Let $\fbar \in \F_q[x_1^{\pm},\dots,x_n^{\pm}]$ be a Laurent polynomial and let $\Vbar \subseteq (\G_m)_{\F_q}^n$ be the toric hypersurface defined by the vanishing of $\fbar$.  Let $\Theta:\F_q \to C$ be a nontrivial additive character (with $C$ a commutative ring of characteristic zero), so that 
\[ \Theta(\xbar+\ybar)=\Theta(\xbar)\Theta(\ybar) \] 
for all $\xbar,\ybar \in \F_q$.  A point of departure for the theory of Dwork is the observation that for $\xbar \in \F_q^{\times n}$, we have 
\[ \sum_{\wbar \in \F_q} \Theta(\wbar\fbar(\xbar)) = \begin{cases} q, &\text{ if $\fbar(\xbar)=0$}; \\ 0, &\text{ otherwise}. \end{cases} \]
Consequently
\begin{equation} \label{qXFq}
q \#\Vbar(\F_q) = \sum_{\substack{\wbar \in \F_q \\ \xbar \in \F_q^{\times n}}} \Theta(\wbar \fbar(\xbar))
= \sum_{(\wbar,\xbar) \in \F_q^{\times (n+1)}} \Theta(\wbar \fbar(\xbar)) + (q-1)^n.
\end{equation}
In other words, counting the set of points $\Vbar(\F_q)$ can be achieved by instead evaluating an exponential sum (on either $\A^1 \times \G_m^{n}$ or $\G_m^{n+1}$).

For $r \in \Z_{\geq 1}$, we define a system of nontrivial additive characters $\Theta_r:\F_{q^r} \to C$ by $\Theta_r=\Theta \circ \Tr_r$ where $\Tr_r:\F_{q^r} \to \F_q$ is the trace map, and we define the exponential sums
\[ S_r(w\fbar, \G_m^{n+1}) = \sum_{(\wbar,\xbar) \in \F_q^{\times (n+1)}} \Theta_r(\wbar\fbar(\xbar)). \]
The $L$-function associated to $w \fbar$ over $\G_m^{n+1}$ is defined to be
\[ L(w\fbar,\G_m^{n+1},T)=\exp\left(\sum_{r=1}^{\infty} \frac{S_r(w\fbar, \G_m^{n+1})}{r} T^r\right). \]
Then by (\ref{qXFq}) we have
\begin{equation} \label{ZetaLZeta}
 Z(\Vbar,qT) = L(w\fbar,\G_m^{n+1},T)Z((\G_m^n)_{\F_q},T), \end{equation}
where
\begin{equation} \label{LbyZeta}
 Z((\G_m^n)_{\F_q},T)^{(-1)^{n+1}} = \prod_{i=0}^{n} (1-q^i T)^{\binom{n}{i} (-1)^i}.
\end{equation}

\subsection*{The Dwork splitting function and interpolation}

Complex characters are defined via the exponential map, but the theory takes off when the ring $C$ where the character takes values is a $p$-adic ring, and the exponential function does not have a large enough $p$-adic radius of convergence to be useful.  We improve this radius of convergence by using a modified exponential function as follows.  Let $\pi$ be an element of the algebraic closure of $\Q_p$ that satisfies $\pi^{p-1}=-p$; then $\Z_p[\pi]=\Z_p[\zeta_p]$ where $\zeta_p$ is a primitive $p$th root of unity.  We define the function
\[ \theta(t) = \exp\left(\pi t+\frac{(\pi t)^p}{p}\right) = \exp(\pi(t-t^p)) = \sum_{i=0}^{\infty} \lambda_i t^i \in \Q_p[\pi][[t]] \]
which is called a \emph{Dwork splitting function}.  It is sometimes denoted by $\theta_1(t)$ to distinguish it from other splitting functions.  We have 
\begin{equation} \label{lambdaest}
\ord_p \lambda_i \geq i (p-1)/p^2,
\end{equation}
where $\ord_p$ is the $p$-adic valuation normalized so that $\ord_p p = 1$; thus, in fact $\theta(t) \in \Z_p[\pi][[t]]$.  We observe that $\theta(1)=1+\pi+O(\pi^2)$ is a primitive $p$th root of unity, and so we obtain our additive characters via the maps 
\begin{align*}
\Theta_r:\F_{q^r} &\to \Z_p[\pi] \\
\Theta_r(\overline{x})&= \theta(1)^{(\Tr \circ \Tr_r)(\overline{x})}
\end{align*}
where $\Tr \circ \Tr_r = \Tr_{\F_{q^r}/\F_p}$ is the absolute trace.

The values of the characters $\Theta_r$ can be $p$-adically interpolated in a way consistent with field extensions, as follows.  Let $\Q_q$ be the unramified extension of $\Q_p$ of degree $a=\log_p q$, and let $\Z_q \subseteq \Q_q$ denote its ring of integers, so that $\Z_q$ is the Witt vectors over $\F_q$.  Let $\sigma:\Z_q \to \Z_q$ denote the $p$-power Frobenius (lifting the $p$th power map on the residue field $\F_q$.)  There is a canonical character $\omega$, called the Teichm\"uller character, of the multiplicative group $\F_q^\times$ taking values in $\Z_q$ that takes an element $\overline{x} \in \F_q^\times$ to the element $x \in \Z_q$, satisfying $x^q=x$ and such that $x$ reduces to $\overline{x}$ in $\F_q$.  For such a Teichm\"uller representative $x \in \Z_{q}$ lifting $\xbar$, we find that
\[ \Theta_1(\overline{x})=\prod_{i=0}^{a-1} \theta(x^{p^i})=\theta(x) \theta(x^p) \cdots \theta(x^{q/p}) \in \Z_p[\pi]. \]
and extending this for $r \in \Z_{\geq 1}$ we have
\[ \Theta_r(\overline{x})=\prod_{i=0}^{ar-1} \theta(x^{p^i}) \in \Z_p[\pi] \]
for $\xbar \in \F_{q^r}$ and $x$ a Teichm\"uller lift of $\xbar$.  
Let $\sigma:\Z_q \to \Z_q$ by $x \mapsto x^{\sigma}$ denote the $p$-power Frobenius, the ring automorphism of $\Z_q$ that reduces to the map $\xbar \mapsto \xbar^p$ modulo $p$.  Then 
\begin{equation} \label{psiranal}
\Theta_r(\overline{x}) = \prod_{i=0}^{ar-1} \theta(x^{\sigma^i}).
\end{equation}

We now consider these character values applied to values of our Laurent polynomial $\fbar$.  Write $\fbar(x)=\sum_{\nu} \overline{a}_\nu x^\nu \in \F_q[x^{\pm}]$ in multi-index notation; we assume that each $\overline{a}_{\nu} \neq 0$.  Let $f(x) = \sum_{\nu} a_\nu x^\nu \in \Z_q[x^{\pm}]$, where $a_\nu$ is the Teichm\"uller lift of $\overline{a}_\nu$.  In light of (\ref{psiranal}), we are led to consider the power series
\[ F(w,x)   = \prod_{\nu} \theta(w a_{\nu} x^\nu) \in \Z_q[\pi][[w,x^{\pm}]] \]
and
\[ F^{(a)}(w,x) = \prod_{i=0}^{a-1} F^{\sigma^i}(w^{p^i},x^{p^i}) \in \Z_q[\pi][[w,x^{\pm}]] \]
where $F^{\sigma}$ denotes the power series obtained by applying $\sigma$ to the coefficients of $F$.  (The abuse of notation which identifies a power series and its specializations will only occur in this paragraph.)
It then follows from (\ref{psiranal}) and a straightforward calculation that
\begin{equation} \label{padicinterpFa}
\Theta_r(\wbar\fbar(\xbar))=F^{(a)}(w,x) F^{(a)}(w^q,x^q) \cdots F^{(a)}(w^{q^{r-1}},x^{q^{r-1}}) \in \Z_p[\pi] 
\end{equation}
for all $(\wbar,\xbar) \in \F_{q^r} \times \F_{q^r}^{\times n}$, where $(w,x)$ denotes the Teichm\"uller lift.  We have thereby extended the interpolation of the values of $\fbar$ to the power series (\ref{padicinterpFa}).

\subsection*{Dwork trace formula}

So far, we have related the zeta function to the $L$-function of an exponential sum via a $p$-adic additive character arising from the Dwork splitting function, and we have interpolated these character values in a power series $F=F(w,x)$.  In order to move this to cohomology, we define a space of $p$-adic analytic functions like $F$ with similar support and $p$-adic growth.

Let $\Delta=\Delta(\fbar)$ be the \emph{Newton polytope} of $\fbar$, the convex hull of its \emph{support} 
\[ \supp(\fbar)=\{\nu \in \Z^n:\overline{a}_\nu \neq 0\}. \]  
For $d \in \R_{\geq 0}$, let $d \Delta=\{dz \in \R^n : z \in \Delta\}$ denote the $d$th dilation of $\Delta$.  Let $L_\Delta$ be the ring
\begin{equation} \label{definitionofL}
L_\Delta=\left\{ \sum_{d=0}^{\infty}\   \sum_{\nu \in d\Delta \cap \Z^n}   c_{d,\nu} w^d x^{\nu} : \text{$c_{d,\nu} \in \Z_q[\pi]$ and $\ord_p(c_{d,\nu}) \geq d\frac{p-1}{p^2}$} \right\}.
\end{equation}
The estimate (\ref{lambdaest}) implies that $F \in L_\Delta$, and so multiplication by $F$ defines a linear operator which we also denote $F:L_\Delta \to L_\Delta$.

On the space $L_\Delta$, we have a ``left inverse of Frobenius'' $\psi:L_\Delta \to L_\Delta$ defined by  
\[ \psi(c_{d,\nu} w^d x^\nu)  =
\begin{cases}
\sigma^{-1}(c_{d,\nu}) w^{d/p} x^{\nu/p}, & \text{ if $p \mid d$ and $p \mid \nu$}, \\
0, & \text{ otherwise};
\end{cases} \]
in multi-index notation, the condition $p \mid \nu$ means $p \mid \nu_i$ for all $i=1,\dots,n$.  The map $\psi$ is $\sigma^{-1}$-semi-linear as a map of free $\Z_q$-modules.

Finally, let $\alpha=\psi \circ F$ and $\alpha_a=\psi^a \circ F^{(a)}$.  Then $\alpha_a$ is $\Z_q$-linear, and another calculation reveals in fact that $\alpha_a=\alpha^a$ (composition $a$ times).

The \emph{Dwork trace formula} \cite{Dwork1} then asserts that 
\[ S_r(w\fbar,\G_m^{n+1})=(q^r-1)^{n+1}\Tr(\alpha_a^r). \]
It follows \cite{ASExpsums} that
\begin{equation} \label{Lchain}
L(w\fbar,\G_m^{n+1}, T)^{(-1)^n} 
= \prod_{j=0}^{n+1} \det(1- (q^j T)\alpha_a^j  \mid L_\Delta)^{(-1)^{j}\binom{n+1}{j}}.
\end{equation}
The equality (\ref{Lchain}) expresses $L(w\fbar,\G_m^{n+1},T)$ via the action of $\alpha_a$ and its powers on an (infinite-dimensional) $p$-adic Banach space; this is the point of departure for Lauder and Wan in their work \cite{LauderWan}.  

We now proceed one step further and move to the level of cohomology.

\subsection*{Dwork cohomology}

We now consider a Koszul complex $\Omega^{\omegadot}$ associated to $wf$ as follows.  To ease notation in this subsection, let $x_0=w$.  
For $i=0,\dots,n$, let 
\begin{equation} \label{fi}
f_i = x_i \frac{\partial (x_0 f)}{\partial x_i}
\end{equation}
and define the operator $D_i:L_\Delta \to L_\Delta$ by
\[ D_i=x_i\frac{\partial}{\partial x_i} + \pi f_i \]
(the latter is the operator given by multiplication by $\pi f_i$).  The operators $D_i$ commute.  For $k=0,\dots,n$, let 
\begin{equation} \label{Omegak}
 \Omega^k=\bigoplus_{0 \leq j_1 < \dots < j_k \leq n} L_\Delta \left( \frac{dx_{j_1}}{x_{j_1}} \wedge \dots \wedge \frac{dx_{j_k}}{x_{j_k}} \right) \cong L_\Delta^{\binom{n+1}{k}}. 
\end{equation}
Let $\Omega^{\omegadot}$ be the complex
\[ 0 \to \Omega^0 \to \Omega^1 \to \dots \to \Omega^{n+1} \to 0 \]
with maps
\[ \nabla\left( \xi \frac{dx_{j_1}}{x_{j_1}} \wedge \dots \wedge \frac{dx_{j_k}}{x_{j_k}} \right) = 
\sum_{i=0}^n (D_i \xi) \frac{dx_{j_i}}{x_{j_i}} \wedge \frac{dx_{j_1}}{x_{j_1}} \wedge \dots \wedge \frac{dx_{j_k}}{x_{j_k}} \]
for $\xi \in L_\Delta$.  

Now $\alpha$ induces a map on the complex $\Omega^{\omegadot}$:
\[
\xymatrix{
0 \ar[r] & \Omega^0 \ar[r] \ar[d]^{p^{n+1}\alpha} & \Omega^1 \ar[r] \ar[d]^{p^{n}\alpha} & \dots \ar[r] & \Omega^{n+1} \ar[r] \ar[d]^{\alpha} & 0 \\
0 \ar[r] & \Omega^0 \ar[r] & \Omega^1 \ar[r] & \dots \ar[r] & \Omega^{n+1} \ar[r] & 0
} \]
since one checks that $\alpha D_i = p D_i \alpha$ for all $i$.  (One similarly has a map induced by $\alpha_a$, replacing $p$ by $q$.)

Then \cite{ASExpsums} we have
\begin{equation} \label{Lcohom}
 L(w\fbar,\G_m^{n+1}, T)^{(-1)^n} = \prod_{j=0}^{n+1} \det(1- \alpha_a T  \mid H^j(\Omega^{\omegadot}))^{(-1)^{n+1-j}}.
 \end{equation}
The condition that $\fbar$ is nondegenerate simplifies the expression (\ref{Lcohom}), as we now see.

\subsection*{Nondegenerate}

We recall our notation from the introduction.  For a face $\tau \subseteq \Delta$, let $\fbar|_\tau=\sum_{\nu \in \tau} \overline{a}_\nu x^\nu$.  Then we say $\fbar$ is \emph{nondegenerate} if for all faces $\tau \subseteq \Delta$ (of any dimension, including $\Delta$ itself), the system of equations
\[ \fbar|_\tau = x_1\frac{\partial \fbar|_\tau}{\partial x_1} = \cdots = x_n\frac{\partial \fbar|_\tau}{\partial x_n} = 0 \]
has no solution in $\overline{\F}_q^{\times n}$, where $\overline{\F}_q$ is an algebraic closure of $\F_q$.

Suppose $\fbar$ is nondegenerate.  Then all the cohomology spaces $H^k(\Omega^{\omegadot})$ are trivial except for $k=n+1$.  Let 
\[ B=\frac{L_\Delta}{D_0 L_\Delta + D_1 L_\Delta + \dots + D_n L_\Delta} \cong H^{n+1}(\Omega^{\omegadot}). \]   
Then by work of Adolphson and Sperber \cite{ASExpsums}, the $\Z_q[\pi]$-module $B$ is free of dimension $n!\vol(\Delta)=v$ (equal to the normalized volume of the cone over $\Delta$ in $\R^{n+1}$) and
\begin{equation} \label{LalphaB}
L(w \fbar,\G_m^{n+1},T)^{(-1)^n} = \det(1-\alpha_a T \mid B) \in 1+T\Z[T]
\end{equation}
(using (\ref{ZetaLZeta})).

Let $A$ (resp.\ $A_a$) be a matrix of $\alpha$ (resp.\ $\alpha_a$) acting on $B$.  Then we have
\begin{equation} \label{alphaafact}
A_a = A A^{\sigma^{-1}} \cdots A^{\sigma^{-(a-1)}}.
\end{equation}
where $A^{\sigma}$ denotes the matrix where $\sigma$ is applied to each entry in the matrix.

\subsection*{An overview of the algorithm}

We now describe how to effectively compute the terms in the formula (\ref{LalphaB}), and in particular the matrix $A$ (\ref{alphaafact}).  We sketch an overview and wait to describe each of these steps and their running time in detail in the sections that follow.

The algorithm takes as input a nondegenerate (confined) Laurent polynomial $\fbar$ and a precision $N \in \Z_{\geq 0}$, and it produces as output the polynomial $\det(1-\alpha_a T \mid B)$ modulo $p^N$.  For $N$ large, we recover the coefficients in $\Z$ and then from (\ref{ZetaLZeta}) we recover $Z(\fbar,T)$.

Let $R=\Z_q/p^N$.  (By carefully factoring out the algebraic element $\pi$, we may work in this smaller ring; see Lemma \ref{ordpw} and the accompanying discussion.)

In (\ref{definitionofL}) we have worked with the power series ring $L_\Delta$, but by the convergence behavior of elements of $L_\Delta$, working modulo $p^N$ these power series become polynomials.  So we define
\[ R[w\Delta] = \bigoplus_{d=0}^{\infty} R[w\Delta]_d \]
where
\[ R[w\Delta]_d = \bigoplus_{\nu \in d\Delta \cap \Z^n} R w^d x^{\nu}. \]
The ring $R[w\Delta]$ is the monoid algebra arising from the cone over $\Delta$ with coefficients in $R$, and it is naturally $\Z_{\geq 0}$-graded by $w$.  Recall (\ref{fi}) that we have defined
\[ f_i=wx_i\frac{\partial f}{\partial x_i} \]
for $i=1,\dots,n$ (and $f_0=wf$).

Our algorithm has 4 steps.

\begin{enumerate}
\item[1.] Compute the Teichm\"uller lift $f$ of $\fbar$.  Using linear algebra over $R$, compute a monomial basis $V$ for the \emph{Jacobian ring}
\[ J  = R[w\Delta]/(wf, wf_1, \dots, wf_n). \]
(The monomial basis $V$ for $J$ yields a basis for $B$.)

\item[2.] For each monomial $m \in V$, compute the action of the Frobenius $\alpha(m)$ using ``fewnomial enumeration''.  

\item[3.] For each $m \in V$, reduce $\alpha(m)$ in cohomology using the differential operators $D_i$ to an element in the $R$-span of $V$.  (The matrices implicitly computed in Step 1 are used in this reduction.)

\item[4.]  Compute the resulting matrix $A=\alpha \mid B$ modulo $p^N$, then compute $A_a$ using (\ref{alphaafact}) and finally 
\[ Z(V, qT)=\det(1-TA_a)^{(-1)^n} Z((\G_m^n)_{\F_q},T). \]  
Output $Z(V, T)$.
\end{enumerate}

\section{The algorithm}

We now describe each step of the algorithm announced in our main theorem and introduced in section 1.  We retain the notation from section 1.

\subsection*{Step 1: Computing the Jacobian ring}

We begin by describing the computation of a basis of the Jacobian ring
\[ J  = \frac{R[w\Delta]}{(wf, wf_1, \dots, wf_n)}; \]
as result of this computation, we also obtain matrices which will be used in the reduction step in Step 3.  

\begin{lemma} \label{nplus1red}
If $f$ is nondegenerate, then $J$ is a free $R$-module with basis of cardinality $v=\Vol(\Delta)$ consisting of monomials with degree $\leq n+1$.
\end{lemma}

\begin{proof}
The proof of Monsky \cite{Monsky} (see work of Adolphson and Sperber \cite[Appendix]{ASExpsums}) shows that under the nondegeneracy hypothesis, the associated Koszul complex is acyclic modulo $p$ and therefore lifts to an acyclic complex over $R$ (as the modules are complete, separated, and flat over $\Z_q$).
\end{proof}

To compute with the ring $R[w\Delta]$, we need some standard algorithms for computing with polytopes.  For a set $S \subseteq \Z^n$, we denote by $\Delta(S)$ its convex hull.

\begin{lemma} \label{enumlatticepointslem}
There exists an efficient algorithm that, given a finite set $S \subseteq \Z^n$, computes the set $\Z^n \cap \Delta(S)$.
\end{lemma}

We discuss this algorithm and its running time in detail in the next section (Proposition \ref{enumeratelatticepoints}); it will be treated as a black box for now.  

Let $\prec$ be a term order on the monomials in $R[w\Delta]$ that respects the $w$-grading.  We begin by computing in each degree $d=0,\dots,n+2$ the set of monomials in $d\Delta$ using Lemma \ref{enumlatticepointslem} and we order them by $\prec$.  Then, for each such $d$ a spanning set for the degree $d$ subspace of the Jacobian ideal, $(wf,wf_1,\dots,wf_n)_d$, is given by the products of the monomials in $(d-1)\Delta$ with the generators $wf,wf_1, \dots, wf_n$ of the Jacobian ideal.  Finally, for each $d$, let $J_d$ be the matrix whose columns are indexed by $\Z^n \cap d\Delta$, i.e.\ the monomial basis for $R[w\Delta]_d$ ordered by $\prec$, and whose rows record the coefficients of the spanning set for $(wf,wf_1,\dots,wf)_d$.  

We then compute the row-echelon form $M_d = T_d J_d$ of $J_d$ for $d=0,\dots,n+2$ using linear algebra over $R$.  According to Lemma \ref{nplus1red}, every pivot in $M_d$ can be taken to be a unit in $R$; thus, a monomial basis $V$ for $J$ is then obtained as $\bigcup_d V_d$ where $V_d$ is a choice of monomial basis for the cokernels of $M_d$.  In particular, the matrix $M_{n+2}$ has full rank and so has a maximal square submatrix with unit determinant.

\subsection*{Step 2: Computing the action of Frobenius}

Recall we have defined the Dwork splitting function
\[ \theta(t)=\exp(\pi(t-t^p)) = \sum_{i=0}^{\infty} \lambda_i t^i \in \Z_p[\pi][[t]] \]
with $\ord_p \lambda_i \geq i (p-1)/p^2$.  The image of $\theta(t)$ modulo $p^N$ in $R[\pi][[t]]$ is a polynomial of degree less than $Np^2/(p-1) = N(p+1+1/(p-1))$.  We compute $\theta(t)=\exp(\pi t)\exp(-\pi t^p)$ as the product of two polynomials of this degree.  

\begin{remark}
Here we have used that $p$ is odd.  For $p=2$, the splitting function $\theta(t)$ above does not converge sufficiently fast for the algorithm described below to work.  For the modifications necessary, see section 5.
\end{remark}

We have $f=\sum_\nu a_\nu x^\nu \in \Z_q[x^{\pm}]$ with $a_\nu \neq 0$ and $\nu \in \Z^n$ and
\begin{equation} \label{Fxexpans}
F(w,x)=\prod_{\nu} \theta(a_\nu w x^\nu).
\end{equation}

One option is to multiply out the product (\ref{Fxexpans}) naively.  Instead, we seek to take advantage of the fewnomialness of $f$.  We have
\begin{equation} \label{prod25}
F(w,x) =\prod_{\nu} \theta(a_\nu wx^\nu) = \prod_\nu \left(1+\lambda_1(a_\nu wx^{\nu}) + \dots + \lambda_j(a_\nu wx^{\nu})^j+\dots\right).
\end{equation}
Let $\supp(f)=\{\nu_1,\dots,\nu_s\}$ and abbreviate $a_i=a_{\nu_i}$.  Expanding (\ref{prod25}) out we obtain
\begin{equation} \label{enumeq_2}
F(w,x) = \sum_{(e_1,\dots,e_s) \in \Z_{\geq 0}^s} 
\left(\lambda_{e_1} \cdots \lambda_{e_s}\right) \left(a_1^{e_1} \cdots a_s^{e_s}\right) w^{e_1+\dots+e_s} x^{e_1\nu_1  + \dots + e_s\nu_s}.
\end{equation}
We make further abbreviations using multi-index notation as follows: for $e \in \Z_{\geq 0}^s$, which we abbreviate $e \geq 0$, we write $\lambda_e=\lambda_{e_1} \cdots \lambda_{e_s}$ and $a^e=a_1^{e_1} \cdots a_s^{e_s}$, and we write $|e|=e_1+\dots+e_s$, and finally $e\nu=e_1\nu_1+\dots+e_s\nu_s$ for the dot product.  Then (\ref{enumeq_2}) becomes simply
\begin{equation} \label{enumeq_2.5}
F(w,x) = \sum_{e \geq 0} \lambda_e a^e w^{|e|} x^{e\nu}.
\end{equation}

Let $w^d x^\mu \in V$.  Then from (\ref{enumeq_2.5}) we have
\begin{equation} \label{enumeq_3}
\alpha(w^d x^{\mu}) = (\psi \circ F)(w^d x^{\mu}) = \sum_{\substack{e \geq 0 \\ p \mid (e\nu+\mu) \\ p \mid (|e|+d)}} \lambda_e \sigma^{-1}(a^e) w^{(|e|+d)/p} x^{(e\nu+\mu)/p} \in L_\Delta. 
\end{equation}
The set of indices for the sum on the right side of (\ref{enumeq_3}) is then contained in the set
\begin{equation} \label{enumeq_4}
E_{d,\mu}=\{e \in \Z_{\geq 0}^s:e_i\nu_i \equiv -\mu_i \psmod{p}\text{ for all $i$ and }|e| \equiv -d \psmod{p}\}.
\end{equation}

When the set $E_{d,\mu}$ is proportionally fewnomial relative to the set of all monomials, we can hope to be able to enumerate it faster than multiplying out $F$ completely.  

\begin{remark}
On the other hand, if $f$ is dense, in general we gain no advantage using this approach, as can be seen by the following example.  Let $f$ be a generic univariate polynomial of degree $s$, so that  $n=1$ and $\supp(f)=\{0,1,\dots,s\}$, and let $\mu \in \Z_{>0}$.  Then there is a bijection between $E_{d,\mu}$ and the set of all (integer) partitions of $d \equiv \mu \pmod{p}$ into parts of size $\leq s$.  Since the number of such partitions grows exponentially with $d$, enumerating all such partitions would be prohibitively time consuming.
\end{remark}

We can compute the set $E_{d,\mu}$ by considering the corresponding set of linear equations modulo $p$.  Let $U$ denote the $(n+1) \times s$ matrix whose columns are the vectors $(1,\nu_i)^t$, and let 
\[ K=K(d,\mu)=\{e:Ue \equiv (d,\mu)^t \psmod{p}\} \subseteq (\Z/p\Z)^s. \]
We identify $K$ with its image in $\Z_{\geq 0}^s$ by taking the smallest nonnegative residue in each component.  Then (\ref{enumeq_4}) becomes simply
\[ E_{d,\mu}=K + (p\Z_{\geq 0})^s. \]
Let $\rho$ be the rank of $U$ modulo $p$.  Then $\# K = p^{s-\rho}$, since by our assumption $\dim(\Delta)=n$ so $s \geq n+1$.

Rewriting (\ref{enumeq_3}), with this notation we obtain
\begin{equation} \label{enumeq_5}
\begin{split}
\alpha(w^d x^{\mu}) &= \sum_{k \in K} \sum_{e \geq 0} \lambda_{k+pe} \sigma^{-1}(a^{k+pe}) w^{(|k|+p|e|+d)/p} x^{\left((k+pe)\nu+\mu\right)/p} \\
&= \sum_{k \in K} \sigma^{-1}(a^k) w^{(|k|+d)/p} x^{(k\nu+\mu)/p} \left( \sum_{e \geq 0} \lambda_{k+pe} a^e w^{|e|} x^{e\nu}\right).
\end{split}
\end{equation}
Here we used the fact that $\sigma^{-1}((a^p)^e) = a^e$ since $a$ is a Teichm\"uller element.  We then compute $\alpha(w^d x^{\mu})$ using the formula (\ref{enumeq_5}).  

We make one final substitution, which will simplify the reduction: we carefully factor out the algebraic element $\pi$ (satisfying $\pi^{p-1} = -p$) as follows.

\begin{lemma} \label{ordpw}
We have $\ord_p \lambda_i \equiv i/(p-1) \in \Q/\Z$.
\end{lemma}

\begin{proof}
The $i$th coefficient of both $\exp(\pi t)$ and $\exp(-\pi t^p)$ satisfy the congruence, and consequently the same is true of the product.
\end{proof}

Therefore, in the expansion $\theta(t)=\sum \lambda_i t^i$ we write $\lambda_i=\pi^i \ell_i$ so that $\ell_i \in \Q_p$ and 
\begin{equation} \label{elldenombounded}
\ord_p(\ell_i) \geq i\left(\frac{p-1}{p^2} - \frac{1}{p-1}\right) = -i\frac{2p-1}{p^2(p-1)}.
\end{equation}
Although we introduce some denominators here, they are controlled.  Extending our multi-index notation, we obtain
\[ \alpha(w^d x^{\mu}) = \sum_{k \in K} \sigma^{-1}(a^k) (\pi w)^{(|k|+d)/p} x^{(k\nu+\mu)/p} \left( \sum_{e \geq 0} \pi^{|k|+p|e|-|e|-(|k|+d)/p} \ell_{k+pe} a^e (\pi w)^{|e|} x^{e\nu}\right). \]
Since $\pi^{p-1}=-p$, we write 
\[ |k|+p|e|-|e|-(|k|+d)/p = (p-1)|e| -d + (p-1)(|k|+d)/p . \]
Thus
\begin{equation} \label{enumeq_6}
\begin{split} 
\alpha((\pi w)^d x^{\mu}) = \sum_{k \in K} \sigma^{-1}(a^k) (-p)^{(|k|+d)/p} &(\pi w)^{(|k|+d)/p} x^{(k\nu+\mu)/p} \\ &\cdot \left( \sum_{e \geq 0} (-p)^{|e|} \ell_{k+pe} a^e (\pi w)^{|e|} x^{e\nu}\right). 
\end{split}
\end{equation}
(We introduce the power $\pi^d$ to simplify this expansion; we may then just divide out at the end.)  

\subsection*{Step 3: Reducing in cohomology}

Our goal in this step is to determine the characteristic polynomial of Frobenius $\alpha$ acting on 
\[ H^{n+1}(\Omega^{\omegadot})=L_\Delta/\textstyle{\sum_{i=0}^n} D_i L_\Delta \]
modulo $p^N$, where $N$ is the desired precision.  The preceding analysis, including Lemma \ref{ordpw} and particularly the expression (\ref{enumeq_6}), shows that $\alpha$ acting on $(\pi w)^d x^\nu$ with $\nu \in d\Delta$ is a series in terms of the form $(\pi w)^e x^{\mu}$ with $\mu \in e\Delta$ and coefficients in $\Z_q$.  The $p$-adic behavior of Frobenius assures us that the coefficients of $(\pi w)^e x^{\mu}$ tend to $0$ in $\Z_q$ as $|e| \to \infty$.  Therefore, modulo $p^N$, the image $\alpha ( (\pi w)^d x^{\nu})$ is a polynomial $G \in R[(\pi w)\Delta]$, a linear combination of terms $(\pi w)^e x^{\mu}$ ($\mu \in e\Delta$) with coefficients in $\Z_q/p^N$.  Since the factor $\pi$ enters only formally, we may suppress the $\pi$ factor and view the resulting polynomial in $R[w\Delta]$.  From Lemma \ref{nplus1red}, we work with coefficients in $\Z_q/p^N$, rather than using the iterative reduction method described by Adolphson and Sperber \cite[Theorem 2.18]{ASExpsums}.

Therefore, let $G \in R[(\pi w)\Delta]$; we show how to reduce $G$ in cohomology.  Let $\lm(G)$ be the leading monomial (highest degree monomial) in $G$ with respect to $\prec$.  First suppose that the degree of $\lm(G)$ (in $w$) is at least $n+2$.  Let $m_0$ be a monomial in $R_{n+2}$ that divides $\lm(G)$ (a precise choice will be given in the next section), and let $m=\lm(G)/m_0$.  Let $G^{(m)}$ consist of those terms in $G$ in $m R[w\Delta]_{n+2}$ and identify $\xi=G^{(m)}/m \subseteq R[w\Delta]_{n+2}$ with a row vector indexed by the monomials of $R[w\Delta]_{n+2}$, each containing the term $(\pi w)^{n+2}$ by our convention.

Let $\eta = \xi T_d$.  Since $J_{n+2}$ is of full rank, we have 
\[ \xi = \xi M_d = (\xi T_d) J_d = \eta J_d. \]
Thus, with the natural identifications, we have written 
\begin{equation} \label{etas!}
\xi= \eta_0 (\pi w) f_0 + \eta_1 (\pi w) f_1 + \dots + \eta_n (\pi w) f_n
\end{equation}
with $\eta_0, \dots, \eta_n \in R[(\pi w)\Delta]_{n+1}= \pi^{n+1} R[w\Delta]$, and therefore
\[ (m\eta_0)(\pi w) f_0 + \dots + (m\eta_n) (\pi w) f_n = m\xi = G^{(m)}. \]
Recall that we work in the module $B=L/\sum_i D_i L$, where $D_i = x_i \partial/\partial x_i + (\pi w) f_i$ (and we again set $w=x_0$ for convenience): this  implies that in $B$, we have the relation
\begin{equation} \label{alphamred}
G^{(m)} \equiv -\left( x_0 \frac{\partial (m \eta_0)}{\partial x_0} + \dots + 
x_n \frac{\partial (m \eta_n)}{\partial x_n}\right) \in B.
\end{equation}
Note now that all terms in the reduction (\ref{alphamred}) have degree one smaller than that of $G^{(m)}$.  We then iterate this procedure on the leading monomial in $G-G^{(m)}$ until it is equivalent in $B$ to a polynomial of degree $\leq n+1$.

So now assume that $G$ has degree $d \leq n+1$, and let $G_d$ be the degree $d$ terms of $G$.  We repeat the above procedure, with $\xi = G_d$: we compute $\eta = \xi T_d$ and let $v_d = \xi - \eta$.  Note that $v_d$ is in the span of $V$ by construction.  Then by a similar calculation as in (\ref{etas!}), we have $\xi = v_d + \sum_{i=0}^{n} \eta_i (\pi w) f_i$ and so 
\[ G_d = \xi \equiv v_d - \sum_{i=0}^{n} x_i \frac{\partial \eta_i}{\partial x_i} \in B. \]
Completing the final iteration on decreasing $d$, we obtain $G \equiv \sum_{d=0}^{n+1} v_d \in B$ written in the span of $V$.  

\begin{remark} \label{nopreclossnodiv}
In this reduction theory, we never perform a division: we simply reduce the power of $\pi w$, as we have written $G$ as a polynomial in $\pi w$ and $x$.  Consequently, we do not lose precision in our analysis.  
\end{remark}

\subsection*{Step 4: Output}

To conclude, we assemble the reductions of $\alpha(w^d x^{\mu})$ for $w^d x^{\mu} \in V$ into a square matrix $A = \alpha \mid B$ of size $v \times v$, where $v=\#V$, with coefficients in $R$.  We then compute 
\[ A_a=A A^{\sigma^{-1}} \cdots A^{\sigma^{-(a-1)}}. \] 
We then compute $\det(1-TA_a)$ using standard methods (analyzed in section 4), and we recover the zeta function from (\ref{ZetaLZeta})--(\ref{LbyZeta}): we compute
\[ Z(\overline{V},qT)=\det(1-TA_a)^{(-1)^n} Z((\G_m^n)_{\F_q},T) \]
from which we easily obtain the desired output $Z(\overline{V},T)$.

\section{Some algorithms for polytopes}

In this section, we describe some methods for computing with polytopes which are used as subroutines.  

\subsection*{Confined polytopes}

By \emph{polytope} we will always mean a lattice polytope.  Let $\Delta \subseteq \R^n$ be a polytope with $\dim \Delta = n$ and normalized volume $v=\Vol(\Delta)$.  (If $\dim \Delta < n$ then by restricting to the linear space that contains $\Delta$ one can make appropriate modifications to the algorithms below.)  The group $\GL_n(\Z)$ acts on $\R^n$ preserving the set of polytopes; we say two polytopes $\Delta,\Delta'$ are \emph{$\GL_n(\Z)$-equivalent} if there exists $U \in \GL_n(\Z)$ such that $U(\Delta)=\Delta'$.  

By work of Lagarias and Ziegler \cite[Theorem 2]{LagariasZiegler}, any polytope $\Delta \subseteq \R^n$ is $\GL_n(\Z)$-equivalent to a polytope contained in a lattice hypercube of side length at most $nv$.  We now prove a slight extension of this result using their methods.

\begin{lemma} \label{LZext}
Any polytope $\Delta \subseteq \R^n$ is $\GL_n(\Z)$-equivalent to a polytope contained in a lattice orthotope (box) with side lengths $b_1,\dots,b_n$ satisfying $b_1 \cdots b_n \leq n^n v$.  
\end{lemma}

\begin{proof}
We follow the proof of Lagarias and Ziegler \cite[Proof of Theorem 2]{LagariasZiegler} (working always with normalized volume); for convenience, we reproduce the main idea of their proof.  First suppose that the polytope is a simplex $\Sigma$ having vertices $v_0,\dots,v_n \in \Z^n$.  Then the lattice $\Lambda$ spanned by the basis vectors $w_i=v_i-v_0$ is a sublattice of $\Z^n$ with $\det(\Lambda)=\Vol(\Delta)=v$.  There is a matrix $U \in \GL_n(\Z)$ taking the matrix $B$ whose column vectors are $w_i$ to its Hermite normal form
\[
UB=\begin{pmatrix} 
a_{11} & 0 & \dots & 0 \\
a_{21} & a_{22} & \dots & 0 \\
\vdots & \vdots & \ddots & \vdots \\
a_{n1} & a_{n2} & \dots & a_{nn} 
\end{pmatrix} \] 
with $0 \leq a_{ji} < a_{ii}$ for $j>i$ and $a_{ii}>0$ for all $i$ and $a_{ji}=0$ for $j<i$.  Now $\det(\Lambda)=|\det(B)|=\prod_{i=1}^n a_{ii} = v$ and hence the parallelopiped generated by the row vectors of $UB$ is contained in the orthotope
\[ \Xi = \{ x \in \R^n : 0 \leq x_i \leq a_{ii} \text{ for $1 \leq i \leq n$}\} = [0,a_{11}] \times \dots \times [0,a_{nn}]. \]
The simplex $U\Sigma$ is contained in this parallellopiped and hence also the translated orthotope $\Xi+Uv_0$.

Suppose $\Delta$ is now an arbitrary polytope.  Then \cite[Theorem 3]{LagariasZiegler} there exists a maximal volume simplex $\Sigma \subseteq \Delta$, and moreover $\Delta$ is contained in the simplex $-n\Sigma + (n+1)z$ where $z=\sum_{i=0}^n v_i$ is the centroid of $\Sigma$ and $(n+1)z \in \Z^n$.  By the above, $\Sigma$ is $\GL_n(\Z)$-equivalent to a simplex contained in an orthotope $\Xi$ and hence $\Delta$ is equivalent to a polytope contained in the orthotope $-n \Xi + (n+1)z$ with side lengths $b_1,\dots,b_n$ with $b_i=n a_{ii}$ so $\prod_{i=1}^n b_i = n^n v$, as claimed.
\end{proof}

We say that $\Delta$ is \emph{confined} if $\Delta$ is contained in an orthotope (box) with side lengths $b_1,\dots,b_n$ with $b_1 \cdots b_n \leq n^n v$.  We say that $\fbar$ is \emph{confined} if $\Delta(\fbar)$ is confined.  

\begin{remark} \label{maxvolsimplex}
The problem of finding a maximum volume simplex, the key to making the proof of Lemma \ref{LZext} computationally effective, has been studied in detail (see, for example, work of Gritzmann, Klee, and Larman \cite{GritzKL}).  To avoid going too far afield into the field of computational geometry, in this article we accept any input Laurent polynomial $\fbar$ but only analyze the runtime when $\fbar$ is confined.  
\end{remark}

\begin{corollary} \label{corcntlattice}
We have $\# (\Delta \cap \Z^n) \leq (2n)^n v$.
\end{corollary}

\begin{proof}
The action of $\GL_n(\Z)$ preserves lattice points, and so if $\Xi$ is the orthotope given by Lemma \ref{LZext} then 
\[ \#(\Delta \cap \Z^n) \leq \#(\Xi \cap \Z^n) = (b_1+1)\cdots (b_n+1) \leq (2b_1)\cdots (2b_n) \leq 2^n (n^n v) \]
as claimed.
\end{proof}

\begin{remark}
Corollary \ref{corcntlattice} is in some sense best possible, since equality holds for $\Delta=[0,1] \subseteq \R$.  For fixed $n$, one can do no better than $\# (\Delta \cap \Z^n) = O(v)$ since for any polytope $\Delta$ we have $\#(d\Delta \cap \Z^n) \sim \Vol(d\Delta)/n!$ as $d \to \infty$.  For a given $n$, one can improve the constant $(2n)^n$ significantly working in a more general context (see Widmer \cite{Widmer}).
\end{remark}

\subsection*{Enumerating lattice points}

Our next task (arising in Step 1) is to exhibit an algorithm to enumerate the lattice points in the convex hull of a set of lattice points.

Let $S \subseteq \Z^n$ be a nonempty finite set.  Let $\Delta=\Delta(S)$ denote the convex hull of $S$.   Let $v=\Vol(\Delta)$ and $s=\# S$.  Suppose $\dim(\Delta) = n$.  Finally, let $\delta=\delta(S)=\max_{\nu \in S} |\nu|$ where $|\nu|=\max_i |\nu_i|$.  Then the bit size of $S$ is $O(s\log \delta)$.  

\begin{proposition} \label{enumeratelatticepoints}
Let $n \in \Z_{>0}$.  Then there exists an algorithm that, given a finite set $S \subseteq \Z^n$, computes the set $\Z^n \cap \Delta(S)$ in time $\widetilde{O}(s^{\lceil n/2 \rceil}\log \delta + v \log \delta)$.  In particular, if $\Delta$ is confined then this runs in time $\widetilde{O}(s^{\lceil n/2 \rceil}+v)$.
\end{proposition}

\begin{remark}
If $\Delta$ is confined, and contained in a orthotope $\Xi$, then $\Delta \cap \Z^n \subseteq \Xi \cap \Z^n$ and the latter has cardinality $O(v)$, but still one needs to test if each such lattice point is contained in $\Delta$.  The exponential contribution of the first term comes from the combinatorics of $\Delta$, which in general can be quite involved.
\end{remark}

The proof of this proposition combines several well-known results in computational geometry.  We refer to the book of Preparata and Shamos \cite{PreSham} and the articles by Seidel \cite{Seidel} and Fortune \cite{Fortune} and the references contained therein for more detail.

The first main step is to compute the \emph{Delaunay triangulation} of $S$.  We lift each point $\nu \in S \subseteq \Z^n$ to $(\nu, \| \nu \|^2) \in \Z^{n+1}$ where $\| \nu \|^2 = \nu_1^2+\dots+\nu_n^2$.  Then the convex hull of the lifted vertices has simplicial faces (under mild assumptions that can be achieved under a suitable perturbation), and projecting the simplicial faces onto the original vertices yields a triangulation.  There are many (deterministic) algorithms to compute the Delaunay triangulation: we will invoke one method, called the \emph{incremental algorithm}, with optimal deterministic variant due to Chazelle \cite{Chazelle}.

To give a brief outline of this algorithm, we follow the overview given by Seidel \cite[19.3.1]{Seidel}.  The incremental algorithm orders the set $S=\{\nu_1,\dots,\nu_s\}$ and incrementally computes $\Delta_i=\Delta(S_i)$ from $\Delta_{i-1}$, where $S_i=\{\nu_1,\dots,\nu_i\}$.  The description of $\Delta_i$ is by its \emph{facet description}, the set of all facets specified by their defining linear inequalities.  A facet of $\Delta_{i-1}$ is \emph{visible} from $\nu_i$ if its supporting hyperplane separates $\Delta_{i-1}$ and $\nu_i$, otherwise we say the facet is \emph{obscured}.  Updating $\Delta_{i-1}$ to $\Delta_i$ involves finding (and deleting) all facets visible to $\nu_i$, preserving all obscured facets, and adding new facets with vertex $\nu_i$.  We record these steps keep track of the triangulation created in this way by updating the \emph{facet graph}.  

From this description, it is clear that the integer operations involve only computing and checking linear inequalities defining facets arising from the convex hull of subsets of points of the form $(\nu,\| \nu \|^2)$ of cardinality $n$.  By Cramer's rule, the linear equalities defining such a facet have coefficients that are bounded in size by $(n+1)! \delta^n (n\delta^2)=O(\delta^{n+2})$, so the largest integer operation can be performed in time $O((n+2)\log \delta)=O(\log \delta)$ using fast integer multiplication techniques.  (The bit arithmetic is also analyzed by Fortune \cite[4.7]{Fortune2}.)

The incremental algorithm requires $O(s \log s + s^{\lfloor (n+1)/2 \rfloor})=\widetilde{O}(s^{\lceil n/2 \rceil})$ integer operations (if $n$ is fixed), so in fact the total number of bit operations is $\widetilde{O}(s^{\lceil n/2 \rceil} \log \delta)$.  

\begin{remark}
Having computed the Delaunay triangulation, in fact we have computed a complete facet description of the convex hull $\Delta=\Delta(S)$; if desired, one can compute the complete face lattice (in particular, the set of all vertices) in the same running time \cite{Seidel}.
\end{remark}

Having computed the Delaunay triangulation, we can list lattice points by doing so in each simplex.  The problem of enumerating lattice points in an (integral) simplex is analyzed by Bruns and Koch \cite{BrunsKoch} in the context of computing the integral closure of an affine semigroup.   Let $\Sigma \subseteq \Delta$ be a simplex of the Delaunay triangulation of $\Delta$ defined by $v_0,\dots,v_n$ and let $w_i=v_i-v_0$.  We compute the Hermite normal form of the matrix $B$ with columns $w_i$ (which can be done using $O(\log \delta)$ bit operations in fixed dimension $n$ \cite{KannanBachem,DomichKT}); let $w_i'$ be its columns.  We then enumerate all elements in the simplex, each represented by a lattice points in the fundamental parallelopiped $\Z^n/(\sum_i \Z w_i')$, using $O(\Vol(\Sigma) \log \delta)$ bit operations.  Since the simplices $\Sigma$ triangulate $\Delta$, putting this together with the first main step, we have proven Proposition \ref{enumeratelatticepoints}.

\section{Precision and running time estimates}

In this section, we discuss precision and running time estimates for each of the four steps of our algorithm.  We suppose throughout this section that $\fbar$ is confined.

\subsection*{Step 1: Computing the Jacobian ring}

First some basics.  A ring operation in $R=\Z_q/p^N$ can be performed using $O(N \log q)$ bit operations using standard fast multiplication techniques.  The Teichm\"uller lift of an element of $\F_q$ to $R$ can be performed using $O(\log N)$ Hensel lift iterations (Newton's method in this case reduces to the iteration of the map $x \mapsto x^q$); each iteration can be performed in time $O(N\log^2 q)$ by repeated squaring, for a total of $O(N \log N \log^2 q)=\widetilde{O}(N \log^2 q)$ bit operations; the time to compute the lift $f$ we will see is negligible.  Similarly, the time to compute the partial deriviatives $wf_i= w \partial f/\partial x_i$ is negligible.  

We now compute the monomial basis $V$ for the Jacobian ring $J$, as described in Section 2.  By Proposition \ref{enumeratelatticepoints}, we can compute $\Z^n \cap d\Delta$ for $d=0,\dots,n+2$ in time $\widetilde{O}(s^{\lceil n/2 \rceil} + v)$: because we have fixed the dimension, we have $\Vol(d\Delta) = d^n \Vol(\Delta) = O(v)$ for such $d$.  (Given the inequalities defining $\Delta$, we could also simply compute $\Z^n \cap (n+2)\Delta$ and then find $\Z^n \cap d\Delta$ by testing the scaled inequalities.)

We now analyze the computation of the row-echelon form $M_d=T_d J_d$ for $d=0,\dots,n+2$.  The matrix $J_d$ has $O(d^n v)$ columns indexed by $\Z^n \cap d\Delta$ and $O( (n+1) (d-1)^n v)=O((d-1)^n v)$ rows indexed by $(wf, wf_1, \dots, wf_n)_d$.  The row echelon form of an $r \times s$-matrix can be performed using $O(r^2 s)$ ring operations using standard techniques, so we can compute $M_d=T_d J_d$ in time
$O( (d-1)^{2n} d^n v^3)$ ring operations and so for $d=1,\dots,n+2$ using
\[ O( n^{3n} v^3 N \log q) = O(v^3 N \log q). \]
bit operations.  From this, we compute the basis $V$.  

\begin{remark}
In practice, it may be more efficient to use Gr\"obner bases to compute the basis $V$ and capture the effect of the reduction matrices $M_d$.  (In particular, the Faug\`ere's $F_4$- and $F_5$-algorithms would be quite useful.)  For simplicity, we take the direct approach using linear algebra.
\end{remark}

\subsection*{Step 2: Computing the action of Frobenius}

First, we compute the image of the Dwork splitting function 
\[ \theta(t) = \exp(\pi(t-t^p)) = \sum_{i=0}^{\infty} \pi^i \ell_i t^i \]
modulo $p^N$.  We have $\ord_p(\pi^i\ell_i)=\ord_p(\lambda_i) \geq i(p-1)/p^2$, so modulo $p^N$ the image is a polynomial of degree less than $Np^2/(p-1) = O(pN)$.  Recall (\ref{elldenombounded}) that
\[ \ord_p(\ell_i) \geq -d(p,i) = -\left\lfloor \frac{i(2p-1)}{p^2(p-1)} \right\rfloor. \]
We write each $\ell_i$ to precision $N$ as an element of the module $p^{-d(p,i)}\bigl(\Z_q/p^{N+d(p,i)} \bigr)$.  The largest such denominator in our expansion is bounded by
\begin{equation} \label{dNp}
d\bigl(p,Np^2/(p-1)\bigr) \leq \left(\frac{Np^2}{p-1}\right) \frac{2p-1}{p^2(p-1)} = N \frac{2p-1}{(p-1)^2} = O(N/p).
\end{equation}
Therefore we compute $\theta(t)$ with coefficients in $\Z_q/p^M$ where $M=N(1+(2p-1)/(p-1)^2)=O(N)$, then compute each $\ell_i=\lambda_i/\pi^i$.

To compute $\theta(t)$ modulo $p^N$, we multiply $\exp(\pi t)$ with $\exp(-\pi t^p)$ truncated to degree $Np^2/(p-1)$.  The latter factor $\exp(-\pi t^p)$ has $\leq Np/(p-1)$ terms, so multiplying the two can be done in time $O((N^2p^3/(p-1)^2)(M\log q)) = O(pN^3 \log q)$.  

\begin{remark}
Note this computation only depends on $p$ and $N$ and does not depend on $f$.
\end{remark}

Given a monomial $w^d x^{\mu} \in V$, we compute the action of Frobenius $\alpha$ using (\ref{enumeq_6}):
\[ \alpha((\pi w)^d x^{\mu}) = \sum_{k \in K} \sigma^{-1}(a^k) (-p)^{(|k|+d)/p} (\pi w)^{(|k|+d)/p} x^{(k\nu+\mu)/p} \left( \sum_{e \geq 0} (-p)^{|e|} \ell_{k+pe} a^e (\pi w)^{|e|} x^{e\nu}\right). \]
This is computed to precision $N$, interpreted as above: to multiply $\ell_i$ times $\ell_j$ we add the exponent of the denominators in $p$ and multiply the numerators as usual; since these denominators are bounded, the extra arithmetic with denominators is negligible and such a multiplication can be performed using $\widetilde{O}(M \log q)=\widetilde{O}(N \log q)$ bit operations.  

Recall Remark \ref{nopreclossnodiv}: in the upcoming reduction step (Step 3), we lose no precision as no divisions occur: we only reduce the power of $\pi w$.  Therefore, to have the Frobenius expansion correct to precision $N$, we just analyze the convergence to $0$ of the term 
\[ c_{k,e} = (-p)^{|e| + (|k|+d)/p} \ell_{k+pe}. \]  
Using (\ref{elldenombounded}), we have
\begin{equation} \label{ckeest}
\begin{split}
\ord_p(c_{k,e}) &= \ord_p\bigl((-p)^{|e| + (|k|+d)/p} \ell_{k_1+pe_1}\cdots \ell_{k_s+pe_s}\bigr) \\
&\geq \frac{|k|+d}{p}+ |e| - (|k|+p|e|)\frac{2p-1}{p^2(p-1)} \\
&\geq (|k|/p+|e|)\left(1-\frac{2p-1}{p^2-p}\right) + \frac{d}{p}.
\end{split}
\end{equation}

Here and from now on we insist that $p>2$: then in (\ref{ckeest}), we have $\ord_p(c_{k,e}) \geq 0$.  In particular, then, the coefficients are all integral (this was guaranteed by the reduction theory of Adolphson and Sperber \cite{ASExpsums}); so even though we have (temporarily) written the values $\ell_{k+pe}$ with denominators, the power of $p$ multiplies through to make them integral.  Also note for any basis element $(\pi w)^d x^{\mu} \in V$, we have $d \leq n+1=O(1)$.  Let 
\[ \beta= \beta(p)=\left(1-\frac{2p-1}{p^2-p}\right)^{-1} = \frac{p^2-p}{p^2-3p+1}. \]
For $p=3$ we have $\beta(p)=6$, but for $p \geq 5$ we have $\beta(p) \leq 20/11$, and $\beta(p) \to 1$ as $p \to \infty$.  In (\ref{enumeq_6}), we have multiplied through by $\pi^d$ for uniformity in the expression, so we must compute to this extra precision: so let
\[ \gamma = \gamma(p,n) = \frac{n+1}{p^2-p} \geq \frac{d}{p^2-p} = \frac{d}{p-1}-\frac{d}{p}. \]
Note that $\gamma(p,n) \leq n$.  

Then (\ref{ckeest}) implies $\ord_p(c_{k,e}/\pi^d) \geq |e|/\beta-\gamma$, since $|k| \geq 0$.  Thus to have the answer correct to precision $N$ we only need to worry about terms with $|e| < \beta(N + \gamma)=E$.  The set of $e \in \Z_{\geq 0}^s$ with $|e| < E$ has cardinality $O(E^s/s!)=O(E^s)$; since $\#K = p^{s-\rho}$, the number of terms in the expansion of $\alpha((\pi w)^d x^\mu)$ is $O(p^{s-\rho} E^s)$.  

We compute the terms in the sum indexed by $K$.  We can compute $a^k=a_1^{k_1}\cdots a_s^{k_s}$ using $O(s(\log|k|)(N\log q))=\widetilde{O}(s N\log^2 q)$ bit operations, since $|k| \leq ps$; $\sigma^{-1}(a^k)=(a^k)^{q/p}$ by repeated squaring is computed using $O(N \log^2 q)$ bit operations.  The exponent arithmetic in the power of $x$ is negligible up to logarithmic factors.  Therefore, we can compute the $\#K=p^{s-\rho}$ terms in the sum indexed by $K$ in time $\widetilde{O}(p^{s-\rho}sN \log^2 q)$.  

Now we compute the terms in the inner sum.  The value $\ell_{k+pe}=\ell_{k_1+pe_1} \cdots \ell_{k_s+pe_s}$ can be computed using $O(s)$ multiplications or $O(s N \log q)$ bit operations.  We compute the values $a^e$ recursively, so that each term (ordered by $\prec$, increasing $|e|$) requires only one additional multiplication.  Therefore the $O(E^s)$ terms in the inner sum can be computed in time $O(E^s (s N \log q))$.  

Putting these together, the total product can be computed using
\begin{equation} \label{timestep2}
\begin{split}
&\widetilde{O}(p^{s-\rho}sN \log^2 q + E^s(sN\log q) + p^{s-\rho}E^s(N \log q)) \\
&\qquad = \widetilde{O}(p^{s-\rho}s E^s N \log^2 q) = \widetilde{O}(p^{s-\rho} \beta^s (N+\gamma)^s N \log^2 q).
\end{split}
\end{equation}
bit operations for one monomial in $V$.

\subsection*{Step 3: Reducing in cohomology}

We now reduce the elements $G=\alpha((\pi w)^d x^{\mu}) \in L$ computed in Step 2.  We begin by reducing the degree of $G$, as explained in Section 2, using multiplication by the matrix $T=T_{n+2}$, until this degree is $\leq n+1$; then we complete the reduction by multiplications by the matrices $T_d$ with $d=n+1,n,\dots,1,0$.  The number of such reductions in any fixed degree $d > n+2$ is governed by the number of translates of $(n+2)\Delta$ that cover $d\Delta \cap \Z^n$.  By definition, we have
\[ d\Delta \cap \Z^n = \bigcup_{\nu \in (d-(n+2))\Delta \cap \Z^n} (\nu + ((n+2)\Delta \cap \Z^n)) \]
so the number of reductions is at most $\#\bigl((d-(n+2))\Delta \cap \Z^n \bigr) = O((d-(n+2))^n v)$ by Corollary \ref{corcntlattice}.  

\begin{remark}
The fewest number of translates that one could hope for is 
\[ \Vol( d\Delta )/\Vol((n+2)\Delta) = (d/(n+2))^n. \] 
But it may be that the combinatorics of $\Delta$ will not allow this.
\end{remark}

Each reduction involves multiplication of a vector by a square matrix of size $O((n+2)^n v) = O(v)$ over $R$, which can be achieved in time $O(v^2 N \log q)$, so reduction from degree $d>n+2$ to $d-1$ takes time $O( (d-(n+2))^n v^3 N \log q)=O(d^n v^3 N \log q)$.  Similarly, reduction from degree $d \leq n+2$ to $d-1$ takes time $O(d^n v^3 N \log q)$.  Repeating this for $d=E,\dots,1,0$ gives a total time of $O( E^{n+1} v^3 N \log q) = O(\beta^{n+1} (N + \gamma)^{n+1} v^3 N \log q)$ to complete the reduction.

\subsection*{Step 4: Output}

Having assembled the square matrix $A$ of size $v$, we compute the product
\[ A_a=A A^{\sigma^{-1}} \cdots A^{\sigma^{-(a-1)}} \] 
(where $a=\log_p q$).  It takes time $O(v^2 N \log^2 q)$ to compute $\sigma^{-1}$ applied to a matrix of size $v$, and time $O(v^3 N \log q)$ to multiply two such matrices, for a total time of $O(\log q (v^2 N \log^2 q + v^3 N \log q)) = O(v^3 N \log^3 q)$ to compute $A_a$.  The characteristic polynomial of a matrix of size $v$ can be computed using $O(v^3 N \log q)$ ring operations, which is absorbed into the previous estimate.

\subsection*{Total running time}

We now add up the contributions from each step, proving Theorem A.  

Step 1 takes time $\widetilde{O}(s^{\lceil n/2 \rceil} + v^3 N \log^2 q)$.  Step 2 takes time $\widetilde{O}(pN^3 \log q + p^{s-\rho} \beta^s (N+\gamma)^s v N \log^2 q)$.  Step 3 takes time $\widetilde{O}( \beta^{n+1}(N+\gamma)^{n+1} v^4 N \log^2 q)$.  Step 4 takes time $O(v^3 N\log q)$.  Since $s \geq n+1$ as $\dim(\Delta) = n$, the time in Step 2 dominates, up to a polynomial in $v N \log q$.  This totals to
\[ \widetilde{O}\bigl(s^{\lceil n/2 \rceil} + p N^3 \log q + p^{s-\rho}\beta^s (N+\gamma)^s (v^4 N \log^2 q)\bigr) \] 
bit operations.  Using the estimate $\beta \leq 6$ and $\gamma \leq (n+1)/20$, this becomes
\begin{equation} \label{thma2runtime}
\widetilde{O}\bigl(s^{\lceil n/2 \rceil} + p N^3 \log q + p^{s-\rho}(6N+n)^s (v^4 N \log^2 q)\bigr).
\end{equation}
and if $p \geq 5$ then $\beta \leq 20/11 \leq 2$ so this improves to
\[ 
\widetilde{O}\bigl(s^{\lceil n/2 \rceil} + p N^3 \log q + p^{s-\rho}(2N+n)^s (v^4 N \log^2 q)\bigr).
\]

\subsection*{Precision}

We have thus computed the characteristic polynomial to precision $N$.  To prove Theorem B, we estimate the value of $N$ required to recover the zeta function itself.  First, we factor this characteristic polynomial so as to work with $p^{-1}\alpha$ instead of $\alpha$.  The matrix $A$ has block form $\begin{pmatrix} 1 & 0 \\ * & * \end{pmatrix}$ where we order the monomials in the basis $V$ for the space $B$ by degree: the only term in degree $0$ is the monomial $1$.  The matrix $A_a$ also has this form, since it holds for each term in the product.  Therefore the characteristic polynomial of $\alpha$ on $B$ factors as $(1-T)$ times the characteristic polynomial on quotient space $B_0=B/R$.  
The action of the Frobenius $\alpha$ on $B_0$ is divisible by $p$; let $A_0$ be the matrix of $\alpha$ on $B_0$.  Then
\[ Z(V,T)=\det(1-(p^{-1})A_0 T)^{(-1)^n}\left(\left.\frac{Z(\G_m^n,T)}{(1-T)}\right|_{T=T/q}\right). \]
Therefore, we may compute with $p^{-1} \alpha$ instead of $\alpha$.

The characteristic polynomial $\det(1-q^{-1} (A_0)_a T)$ then has inverse roots of absolute value at most $q^{n/2}$ by a theorem of Adolphson and Sperber \cite{ASExpsums} and Denef and Loeser \cite{DenefLoeser}.  Therefore its $i$th coefficient is bounded by $\binom{v}{i} q^{i(n/2)}$.

\begin{lemma} \label{stupidlemma}
For all $x \in \R_{\geq 1}$ and $v \in \Z_{\geq 0}$, we have 
\[ \max_{0 \leq i \leq v} \binom{v}{i} x^i = \binom{v}{\lceil v/2 \rceil+j} x^{\lceil v/2 \rceil+j} \] 
where $0 \leq j \leq \lfloor v/2 \rfloor$ is the unique index such that
\[ \frac{\lceil v/2 \rceil+j}{\lfloor v/2 \rfloor-(j-1)} \leq x < \frac{\lceil v/2 \rceil+(j+1)}{\lfloor v/2 \rfloor-j}. \]
\end{lemma}

The proof is an easy inductive argument.  In the two extremes: if $x \geq v$ then $j=\lfloor v/2 \rfloor$ and the largest coefficient is $x^v$; if $x < (\lceil v/2 \rceil + 1)/\lfloor v/2 \rfloor$ (equal to $1+2/v$ if $v$ is even, for example) then the largest coefficient is $\binom{v}{\lceil v/2 \rceil} x^{\lceil v/2 \rceil}$.  It follows that the $p$-adic precision $N$ required to recover all of these coefficients as integers from their reduction modulo $p^N$ is given by
\begin{equation} \label{padicestimate}
p^N \geq 2 \binom{ v}{\lceil v/2 \rceil + j} (q^{n/2})^{\lceil v/2 \rceil + j}
\end{equation}
where $j$ is given as in Lemma \ref{stupidlemma} with $x=q^{n/2}$.

In practice, one will want to work with the precision estimate (\ref{padicestimate}).  To estimate the runtime, we have the crude bound
\[ \binom{v}{i} x^i \leq \binom{v}{\lceil v/2 \rceil} (q^{n/2})^v < (2q^{n/2})^v \]
which implies we may take
\begin{equation} \label{Ncheapest}
N \geq (v+1)\log_p 2 + \frac{nv}{2}\log_p q = O(nv\log q).
\end{equation}

\begin{remark}
We are forced to take a larger bound than just the middle coefficient because we do not have a Riemann hypothesis in this generality.  For many varieties under consideration, such a hypothesis will give a better estimate on the precision, since the higher coefficients are determined by the lower ones.  

Also, work of Kedlaya \cite{KedlayaRootUnitary} shows how to recover the zeta function often in practice with much less precision knowing only that it can be factored as a product of Weil $q$-polynomials.
\end{remark}

Plugging the estimate (\ref{Ncheapest}) into (\ref{thma2runtime}), and considering $s$ (and $n$) to be fixed, we obtain the estimate
\[ \widetilde{O}(p (v\log q)^3 \log q + p^{s-\rho} (v\log q)^s v^4 (v\log q)\log^2 q) = 
\widetilde{O}(p^{\min(1,s-\rho)} v^{s+5} \log^{s+3} q) \]
for the number of bit operations performed.  This completes the proof of Theorem B.

\section{Modifications}

In this section, we discuss some extensions and modifications to the above algorithm.  

\subsection*{Dense input}

We can also modify the algorithm for the situation of dense input.  One can also forget the condition of sparsity and analyze the running time on dense input.  Here, we do not use the expansion (\ref{enumeq_6}), but rather directly compute the product (\ref{Fxexpans}).  The analysis above shows that the computation of the expansion 
\[ \theta(a_\nu w x^\nu) = \sum_{i=0}^{Np^2/(p-1)} \ell_i (\pi w)^i a_\nu^i x^{i\nu} \]
can be performed in $\widetilde{O}(spN^2 \log q)$ operations.  We have $s$ such terms in the product (\ref{Fxexpans}).  This product has monomial support in $(Np^2/(p-1))\Delta$ so has $O((pN)^n v)$ terms.  

As in (\ref{dNp}), we compute in $\Z_q/p^M$, and multiplying two polynomials in $n$ variables with coefficients in $R=\Z_q/p^M$ with at most $O((pN)^{n} v)$ terms takes time $O(M(pN)^{2n} v^2 \log q)$ if multiplied term-by-term.  (See also Lauder and Wan \cite[Lemma 30]{LauderWan}.)  (The grading by $w$ allows us to multiply more carefully, but this saves only a constant factor for fixed dimension.)  Therefore, the product $F(w,x)$ can be computed in time
\[ \widetilde{O}(spN^2 \log q + sM(pN)^{2n} v^2 \log q) = \widetilde{O}(sp^{2n} N^{2n+1} v^2 \log q). \]
Therefore, one sees a benefit from the fewnomial expansion only when $s \leq 2n+1$.  

The time to compute $\alpha(w^d x^{\mu})$ requires negligible time in comparison, as it involves only exponent arithmetic and applying the inverse Frobenius $\sigma^{-1}$.  The other steps are unmodified, so the total time is
\[ \widetilde{O}\bigl(s^{\lceil n/2 \rceil} + sp^{2n}N^{2n+1} v^2 \log q + (6N+n)^{n+1} v^3 N \log q)\bigr) 
= \widetilde{O}\big(s^{\lceil n/2 \rceil} + sp^{2n}N^{2n+1} v^3 \log q) \]
to compute the zeta function modulo $p^N$ (the analogue for Theorem A), and plugging in the estimate (\ref{Ncheapest}) for $N$ and considering $s$ fixed we obtain
\[ \widetilde{O}(p^{2n} v^{2n+4} \log^{2n+2} q) \]
(for Theorem B).

\subsection*{Modifications when $p=2$}

For many applications (notably in coding theory), the computation of zeta functions and $L$-functions in characteristic $2$ are particularly important.  While it may be possible to perform this analysis, we do not do so in the present work.  Instead, we mention some of the hurdles this analysis faces in characteristic $2$ using the approach we have taken applying Dwork cohomology.  (Some of these same hurdles, and others, also arise in other $p$-adic approaches.)

Dwork's original $p$-adic study of zeta functions concerned a nonsingular projective hypersurface $V \subset \PP^{n-1}$ over $\F_q$ defined by the vanishing of a homogeneous form $f(x) \in \F_q[x_1,\dots,x_n]$ of degree $d$.  When $\gcd(2,p,d)=1$, Dwork constructs a $p$-adic cohomology space with an action of Frobenius such that the characteristic polynomial of Frobenius acting on cohomology gives the important nontrivial middle dimensional primitive factor of $Z(V,T)$.  In particular, if $p=2$, Dwork's cohomology does not at present apply to smooth hypersurfaces of even degree in characteristic $2$.  

In some cases, Adolphson and Sperber \cite{ASnondegwrtlat} are able to supplement Dwork's work when $p=2 \mid d$.  For example, if $p \mid n$ and we consider the Dwork family of hypersurfaces
\[ f(x_1,\dots,x_n)=x_1^n+x_2^n+\dots+x_n^n-\lambda x_1,\dots,x_n \]
in characteristic $p=2$, then even though the family consists of singular hypersurfaces, they are nondegenerate with respect to the sublattice of $\Z^{n+1}$ generated by the support of $wf(x)$.  As a consequence (even when $p=2$), there is a cohomology space such that the characteristic polynomial of Frobenius acting on this space is the nontrivial factor of the zeta function.  But computing in the sublattice adds to the computational complexity.  

A second obstacle concerns convergence and our choice of splitting function
\[ \theta_1(t)=\exp(\pi t-(\pi t)^p/p) \]
where $\pi^{p-1}=p$, which converges for $\ord_p t > -(p-1)/p^2$.  This led us to the space $L_\Delta=L_\Delta((p-1)/p^2)$ consisting of power series with similar growth) and the operators $D_i=x_i (\partial/\partial x_i) + \pi x_i w (\partial f)/(\partial x_i)$ acting on $L_\Delta$.  Our explicit reduction theory in the Jacobian ring depended on the operator norm of $x_i (\partial/\partial x_i)$ being $p$-adically smaller than the operator norm of $\pi x_i w (\partial f)/(\partial x_i)$, so that the series (\ref{enumeq_6}) converges after reduction.  This requires that $1/(p-1) \leq (p-1)/p$, i.e.\ $p<(p-1)^2$, and this fails if and only if $p=2$.

Therefore, when $p=2$, it is necessary to use a splitting function with better convergence properties.  Let $\pi \in \overline{\Q}_2$ be a nonzero root of the equation
\[ \pi^8/8 + \pi^4/4 + \pi^2/2 + \pi=0, \quad \text{i.e.,} \quad \pi^7+2\pi^3+4\pi+8=0 \]
and let
\[ \theta_3(t)=\exp\left(\pi t + \frac{(\pi t)^2}{2} + \frac{(\pi t)^4}{4} + \frac{(\pi t)^8}{8} \right) = \sum_{i=0}^{\infty} \lambda_i t^i \in \Z_2[\pi][[t]]. \]
In this case we have $\ord_2 \lambda_i > (11/16)i$.  Working with the space $L_\Delta(11/16)$, reduction is possible (in $L_\Delta(11/8)$) since $11/8>1/(p-1)=1$ when $p=2$.

The computational difficulties caused by using $\theta_3(t)$ instead of $\theta_1(t)$ are numerous.  Not only is the reduction algorithm more difficult, but even the calculation of Frobenius is more complicated \cite[Proposition 3.2]{ASExpsums}.

\begin{remark}
Robba constructed a constant called $\pi_{\textup{Robba}}$ by Dwork with $\ord_p \pi_{\textup{Robba}}=1/(p-1)$ with splitting function $\theta_{\textup{Robba}}(t)=\exp(\pi_{\textup{Robba}}(t-t^p))$ which has a better radius of convergence than Dwork's $\theta_1(t)$.  However, even with this improvement, the $p$-adic norm in case $p=2$ of $\pi_{\textup{Robba}} x_i w (\partial f)/(\partial x_i)$ does not dominate the norm of $x_i (\partial/\partial x_i)$ on the appropriate Banach space.
\end{remark}

For these reasons, we do not deal at present with the case $p=2$.

\subsection*{Affine varieties}

Next, we describe modifications to the algorithm to compute the zeta function of affine varieties.  (At the price of some additional notation, one could consider the more general case where the variety is a combination of affine and toric.)

Let $\fbar(x)=\sum_\nu \overline{a}_\nu x^{\nu} \in \F_q[x_1,\dots,x_n]$ be a polynomial.  For a subset $A \subseteq [n]=\{1,\dots,n\}$, we denote $\fbar_A(x) \in \F_q[x_i : i \in [n] \setminus A]$ the polynomial obtained from $\fbar$ obtained by setting $x_i=0$ for all $i \in A$.  We say that $\fbar$ is \emph{convenient} (\emph{with respect to the variables} $x_1,\dots,x_n$) provided 
\[ \dim \Delta(\fbar_A) = \dim \Delta(\fbar) - \#A \]
for all subsets $A \subseteq [n]$.  Equivalently, $\fbar$ is convenient if and only if $\fbar$ has a nonzero constant term and a monomial $x_i^{d_i}$ with $d_i \in \Z_{>0}$ for all $i=1,\dots,n$.  The notion of convenient is also called \emph{commode}.  

We suppose for the rest of this subsection that $\fbar$ is convenient and nondegenerate (with respect to $\Delta(\fbar)$, defined as before).  It is a consequence of the hypothesis of convenience that $\dim \Delta(\fbar)=n$.

\begin{example}
Suppose $\fbar$ has total degree $d$ and $\Delta(\fbar)$ is the convex hull of the set of points $\{0\} \cup \{d e_i\}^n$ where $e_1,\dots,e_n$ is the standard basis in $\R^n$.  Then we may write $\fbar = \fbar^{(d)} + \overline{g}$ where $\fbar^{(d)}$ is the form of $f$ of highest degree terms and the total degree of $\overline{g}$ is less than $d$.  Then $\fbar$ is nondegenerate if and only if $\fbar_A^{(d)}=0$ defines a nonsingular projective hypersurface in $\PP^{n-\#A-1}$ and $\fbar_A=0$ is nonsingular in $\A^{n-\#A}$ for all subsets $A \subseteq [n]$.  
\end{example}

Assuming then that $\overline{f}$ is nondegenerate and convenient, we can work more simply with the affine $L$-function $L(wf, \G_m \times \A^n,T)$, as follows.  Let $\overline{V}$ denote the affine hypersurface in $\A^n_{\F_q}$ defined by $\fbar=0$.  We modify the calculation in Section 1 by working with an affine exponential sum:
\[ q^{rn} + \sum_{\substack{w \in \F_{q^r}^\times \\ x \in \F_{q^r}^n}} \Theta_r(wf) = \sum_{w,x \in \F_{q_r} \times \F_{q_r}^n} \Theta_r(wf) = q \# \overline{V}(\F_{q^r}) \]
so
\[ Z(\overline{V},qT)=\frac{L(w\fbar, \G_m \times \A^n, T)}{1-q^nT}. \]

The computation of $L(w\fbar, \G_m \times \A^n, T)$ is for the most part quite similar to the toric calculation given earlier.  We describe here the required modifications.  Let $A \subseteq S=\{1,\dots,n\}$.  Let $L_\Delta^{(A)}$ denote the ideal in the ring $L_\Delta$ consisting of series having support in monomials divisible by $x_A=\prod_{i \in A} x_i$.  Under our hypotheses, the complex $\Omega^{\omegadot}$ for the $L$-function has vanishing cohomology $H^i(\Omega^{\omegadot})$ for all $i$ except $i=n+1$ and 
\begin{equation} \label{Baffine}
H^{n+1}(\Omega^{\omegadot})=\frac{L_\Delta^S}{D_0 L_\Delta^S + \sum_{i=1}^n D_i L_\Delta^{S \setminus \{i\}}}
\end{equation}
where the $D_i$ are as in section 1.  Note that here $H^{n+1}(\Omega^{\omegadot})$ is contained in the cohomology space $L_{\Delta}/\left(\sum_{i=0}^n D_i L_{\Delta}\right)$.  

The required modifications are then simple.  For example, the Jacobian ring has the form
\[ J^S = \frac{R[w\Delta]^S}{(wf)R^S + \sum_i w f_i R^{S \setminus i}}. \]
The reduction algorithm then only requires identification of monomials in the ideals $L_\Delta^S$ and $L_\Delta^{S \setminus \{i\}}$, and the recursive reduction process preserves the required divisibilty on monomials so the algorithm runs in every other way without modification.  

In this way, one can reduce the sizes of computations involved under the convenient hypothesis: one avoids calculation of the contributions to the zeta function coming from the intersection of the affine hypersurface with the coordinate hyperplanes.  

\subsection*{Projective varieties}

We now consider the modifications for projective varieties.

Suppose that $\fbar \in \F_q[x_1,\dots,x_n]$ is a homogeneous form of degree $d$ with $\gcd(p,d)=1$ and that $\fbar=0$ defines a nonsingular projective hypersurface $\overline{Z} \subseteq \PP^{n-1}_{\F_q}$.  Suppose further that $\fbar_A=0$ defines a nonsingular projective hypersurface in $\PP^{n-\#A-1}$ for all subsets $A \subseteq [n]$; such a hypersurface is said to be in \emph{general position}.  If $\gcd(p,d)=1$, then $\fbar$ defines a projective hypersurface in general position if and only if $\fbar$ is nondegenerate (with respect to $\Delta(\fbar)$) and convenient (with respect to the variables $x_1,\dots,x_n$).  

Here, the support of $w\fbar$ lies in the hyperplane $\sum_{i=1}^n x_i=dw$ in $\R^{n+1}$, so $\dim \Delta(w\fbar)=n$ (or, equivalently, $\dim \Delta(\fbar)=n-1$).  

It is well-known that
\[ Z(\overline{Z},T)=\frac{P(T)^{(-1)^{n-1}}}{(1-T)(1-qT)\cdots (1-q^{n-2}T)} \]
where $P(T)$ is a polynomial of degree $d^{-1}\bigl((d-1)^n +(-1)^n(d-1)\bigr)$ that represents the action of Frobenius on middle-dimensional primitive cohomology.  By work of Adolphson and Sperber \cite{ASExpsums}, the cohomology of the complex $\Omega^{\omegadot}$ for $L(w\fbar, \G_m \times \A^n,T)$ described above in this case has vanishing cohomology $H^i(\Omega^{\omegadot})=0$ for $i\neq n,n+1$ and that there is an isomorphism of Frobenius modules $H^{n+1}(\Omega^{\omegadot}) \cong H^n(\Omega^{\omegadot})$ with Frobenius on $H^n$ being $q$ times the Frobenius on $H^{n+1}$.  As a consequence, 
\[ \det\bigl(1-\Frob T \mid H^{n+1}(\Omega^{\omegadot})\bigr) = P(qT) \]
yielding the ``interesting'' part of the zeta function of $\overline{Z}$.  

As in the affine case, the space $H^{n+1}$ is isomorphic to the quotient defined in (\ref{Baffine}).  But we can simplify further.  By the Euler relation, we have $dD_0 = \sum_{i=1}^n D_i$ on $L_\Delta$ so that
\[ D_0 L_\Delta^S \subseteq \sum_{i=1}^n D_i L_\Delta^{S \setminus \{i\}}. \]
This enables us to reduce the calculation by suppressing the role played by $w$ entirely.  To determine the monomials $x^\nu$ in $L_\Delta$, we simply need to check 
\[ x^\nu \in M_\Delta = \left\{ \nu \in \Z_{\geq 0}^n : |\nu|=\textstyle{\sum}_{i=1}^n \nu_i \equiv 0 \pmod{d} \right\} \]
and use $w^{|\nu|/d}x^{\nu}$.  So the power of $w$ enters only formally.  

Let 
\[ \widetilde{D}_i = x_i \frac{\partial}{\partial x_i} + \pi x_i \frac{\partial}{\partial x_i} f \]
and write
\[ \widetilde{L}_\Delta = \left\{ \sum_{\nu \in M_\Delta} a_\nu x^\nu : a_\nu \in \Z_q[\pi], \ord_p a_\nu \geq \frac{|\nu|}{d} \left(\frac{p-1}{p^2}\right)\right\}\subset \Z_q[\pi][[x]] \]
Our object of interest for the reduction is then
\[ B'' = \frac{\widetilde{L}_\Delta^S}{\sum_{i=1}^n \widetilde{D}_i \widetilde{L}_\Delta^{S\setminus\{i\}}}. \]
The preceding algorithms for reduction may be applied here as well; the appropriate powers of $w$ may be formally inserted as necessary.

Finally, as we remarked in the comments for the case $p=2$, when $p \mid d$ there are further modifications that can be made by considering polynomials that are nondegenerate relative to a sublattice even in some exceptional singular cases \cite{ASnondegwrtlat}.

\subsection*{Exponential sums}

In Section 1, we reduce the problem of computing the zeta function to the problem of computing the $L$-function of an exponential sum.  But in many situations the problem of computing this $L$-function itself is of interest.  In this case, there is no auxiliary or dummy variable $w$; the support of the $p$-adic power series in our space $L_{\Delta}$ consists of those lattice points in the cone over $\Delta_{\infty} (\overline{f})$ which itself is the convex closure of the support of $\overline{f}$ together with the origin.  We note the earlier work of Lauder and Wan \cite{LauderWanArtinSchreier,LauderWanArtinSchreierII} who apply some similar approach to compute the $L$-function of a one-dimensional exponential sum (i.e., $n=1$).  

We consider the case of toric exponential sums.  (Considering affine exponential sums, when convenient, requires modifications similar to those above.)  Let $\fbar(x)=\sum_{\nu} \overline{a}_\nu x^{\nu} \in \F_q[x^{\pm}]$ be a Laurent polynomial.  We say $f$ is \emph{quasihomogeneous} if there are rational numbers $\alpha_1,\dots,\alpha_n$ such that $w(\nu)=1$ for all $\nu \in \supp(\fbar)$ where 
\[ w(\nu)=\sum_{i=1}^n \alpha_i \nu_i. \]
  We restrict to the case of quasihomogeneous exponential sums as this conforms quite closely with our computation of zeta functions; the method could be adapted to the general case.  

There are fewnomial examples of quasihomogeneous exponential sums that indeed appear nontrivial.  For example, consider a subset $\Lambda=\{\nu^{(1)}, \dots, \nu^{(n+1)}\} \subseteq \Z^n$ of cardinality $n+1$ with each element $\nu^{(i)} \in \Lambda$ lying on the hyperplane $H$ defined by $\sum_i \alpha_i x_i = 1$ for all $i=1,\dots,n+1$.  When the convex hull $\Delta$ of $\Lambda$ is not an $n$-simplex, then the $L$-function associated with
\[ \fbar(x)=\sum_{i=1}^{n+1} a_i x^{\nu^{(i)}} \]
on $\G_m^n$ is not well-understood.  Even when $\Delta$ is an $n$-simplex, such sums are not entirely understood.  We know from Adolphson and Sperber \cite{ASExpsums} that if $\overline{f}$ is nondegenerate with respect to $\Delta_{\infty}(\overline{f})$, then the complex $\Omega^{\omegadot}(\fbar,\G_m^n)$ for this $L$-function is acyclic except in dimension $n$ and
\[ H^n(\Omega^{\omegadot}(\fbar,\G_m^n)) = \frac{L_\Delta}{\sum_{i=1}^n D_i L_\Delta}. \]

Now $L(\overline{f},\G_m^n,T)^{(-1)^{n+1}} \in \Z[\zeta_p][T]$.  This is one complication.  The second complication is that the reduction has a modification: now we bring things down by weight.  This is analogous to the degree of $w$, and the argument is formally the same.  The ring is still graded by the weight.  

If all vertices lie in a hyperplane (\emph{quasi-homogeneous}), then the Frobenius matrix has the property that the elements belong to $\pi^{\NN} \Z_q$: the $\pi$ and the weight move as one unit.  (So we can do multiplications in the smaller ring.)  Then the filtered ring is a graded ring.  The only time when one has to do honest calculations in $\Z[\zeta_p]$ is in the final calculation of the characteristic polynomial.

\subsection*{Twisted exponential sums on the torus}

Let $\overline{f}(x) \in \F_q[x^{\pm}]$ be nondegenerate.  Let $\chi_1,\dots,\chi_n$ be multiplicative characters of $\F_q$ (possibly including the trivial character).  Since the character group of $\F_q^\times$ is generated by the Teichm\"uller character $\omega$, each $\chi_i$ may be identified with an integral power $a_i$ of $\omega$ with $0 \leq a_i < q-1$.  It is useful to write $\chi_i=\omega^{a_i}=\omega^{\gamma_i(q-1)}$ where $\gamma_i=a_i/(q-1) \in [0,1)$.  The shifted lattice $\Lambda(\gamma)=(\gamma_1,\dots,\gamma_n) + \Z^n$ plays an important role in the cohomological study of the twisted sums
\[ S(\gamma, \overline{f}, \G_m^n) = \sum_{x \in \G_m^n(\F_q)} \omega^{a_1}(x_1)  \cdots \omega^{a_n}(x_n) \Theta(\overline{f}(x_1,\dots,x_n)) \]
and the associated $L$-function $L(\gamma,\overline{f},T)$.

We modify now our earlier work on quasihomogeneous nondegenerate toric sums to include the case of twisted sums.  For $\nu \in \Z^n$ we define $w(\nu)$ to be the smallest $m \in \Q_{\geq 0}$ such that $\nu \in m\Delta$.  We define
\[ L_\Delta=\left\{ \sum_{\nu \in M_\gamma(\overline{f})}   c_{\nu} x^{\nu} : \text{$c_{\nu} \in \Z_p[\zeta_{(q-1)p}]$ and $\ord_p(c_{\nu}) \geq \frac{p-1}{pq}w(\mu)$} \right\} \]
where $M_\gamma(\overline{f})$ is the intersection of the cone over $\overline{f}$ intersected with $\Lambda(\gamma)$.  

As before, let $f$ denote the Teichm\"uller lift of $\overline{f}$.  The Frobenius map $\alpha=\psi_q \circ \exp(\pi f(x) - f(x^q))$ acts on $L_\Delta$ and the Dwork trace formula for $L(\gamma,\overline{f},T)$ takes the form
\[ L(\gamma,\overline{f},T)^{(-1)^{n+1}} = \det(1-T\alpha \mid L_\Delta^{(\gamma)})^{\delta^n} \]
where $g(T)^\delta=g(t)/g(qT)$ for $g \in \C_p[[T]]$.  As usual, we define differential operators
\[ D_i = x_i \frac{\partial}{\partial x_i} + \pi x_i \frac{\partial f}{\partial x_i} \]
on $L_\Delta$, and construct a complex $\Omega\omegadot$ using $L_\Delta$ as the base and boundary operator as before.  The Frobenius defines a chain map on this complex using $\alpha$, and the hypothesis that $f$ is nondegenerate implies that $H^i(\Omega\omegadot)=0$ for $i \neq n$ and
\[ L(\gamma,\overline{f},T)^{(-1)^{n+1}}=\det(1-\alpha T \mid H^n(\Omega\omegadot)) \]
where
\[ H^n(\Omega\omegadot)=\frac{L_\Delta}{\sum_{i=1}^n D_i L_\Delta} \]
is a free $\Z_p[\zeta_{(q-1)p}]$-algebra of finite rank $\Vol(\Delta)$.  Our calculation of the matrix of Frobenius acting on $H^n(\Omega\omegadot)$ follows the same argument used above in the case of (``untwisted'') quasi-homogeneous nondegenerate toric sums.

\subsection*{Multiplicative character sums on the torus}

Continuing with this line of analysis, let $\chi=\omega^{a_0}=\omega^{\gamma_0(q-1)}$ be a nontrivial multiplicative character of $\F_q$.  We extend $\chi$ to all of $\F_q$ by setting $\chi(0)=0$.  As in the previous section, let $\overline{f} \in \F_q[x]$ be nondegenerate, and consider the character sum
\[ S(\gamma_0,\fbar,\G_m^n)=\sum_{x \in \G_m^n(\F_q)} \omega^{\gamma_0(q-1)}(\overline{f}(x)) \]
and its associated $L$-function $L(\gamma_0,\overline{f},T)$.  

We now use an elementary character argument to transform such a multiplicative character sum to a twisted exponential sum of the type considered in the previous section.  Suppose that $\chi$ is nontrivial (i.e., $\gamma_0 \neq 0$).   Then for $u \in \F_q^\times$ we have 
\[ \sum_{t \in \F_q^\times} \chi^{-1}(t)\Theta(tu)
= -(G(\chi^{-1},\Theta))\chi(u) \]
where
\[ G(\chi^{-1},\Theta)=-\sum_{v \in \F_q^\times} \chi^{-1}(v)\Theta(v) \]
is the negative of a Gauss sum.  Since $\chi$ is nontrivial, this identity holds for $u=0$ as well.  Therefore our sum of interest
\[ S(\gamma_0,\overline{f},\G_m^n) = -G(1-\gamma_0,\Theta)^{-1} S(\gamma,w\overline{f},\G_m^{n+1}) \]
where $\gamma=(1-\gamma_0,0,\dots,0) \in \Q^{n+1}$ and the exponential sum on the right is a twisted sum of the type in the preceding section.

By the Hasse-Davenport relation on Gauss sums, we have
\[ L(\gamma_0,\overline{f},\G_m^n, G(\chi^{-1},\Theta)T)^{-1} = L(\gamma,w\overline{f},\G_m^{n+1},T). \]
Note $w\overline{f}$ is always quasihomogeneous, as the monomials all lie in the hyperplane in $\R^{n+1}$ with equation $w=1$.  If $f$ is nondegenerate, then we may proceed to compute the $L$-function as in the previous section on twisted sums.

\section{Examples}

\subsection*{Elliptic curve point counting}

In this subsection, we give an example to show how our methods can be used to compute the zeta function of an elliptic curve.  In this situation, our method is not competitive with more specialized methods, but running through the algorithm in this case will hopefully shed some insight on how it works.

Let $p \geq 3$ be prime.  Let $\fbar=x^3+\overline{a}x+\overline{b}-y^2 \in \F_q[x,y]$ be such that $4\overline{a}^3+27\overline{b}^2 \neq 0$, so that $\fbar=0$ defines an affine piece of an elliptic curve $\overline{E}$ over $\F_q$.  (This does not cover all cases when $p=3$; we leave the other examples to work out by the interested reader.)  Let $f=x^3+ax+b-y^2$ be the Teichm\"uller lift of $f$ to $\Z_q[x,y]$.  

If $b=0$, then $\overline{E}$ has complex multiplication by $\Z[i]$ and is a twist of the elliptic curve $y^2=x^3-x$, so the zeta function can easily be recovered by classical methods; the same is true if $a=0$.  So we assume that $ab\neq 0$.  Therefore, the polytope $\Delta$ is the triangle given by the convex hull $\Delta(\{(0,0),(3,0),(0,2)\})$.  One can check that $f$ is automatically nondegenerate given the nonvanishing of the discriminant $4\overline{a}^3+27\overline{b^2} \neq 0$.  (See also work of Castryck and the second author \cite{CV}.)  Furthermore $f$ is convenient (with respect to $x,y$).

In this situation, we have
\[ L(w\fbar, \G_m \times \A^2, T) = P(qT) \]
where
\[ Z(\overline{E},T)=\frac{P(T)}{(1-T)(1-qT)}=\frac{1-a_q T + qT^2}{(1-T)(1-qT)} \]
and $a_q=q+1-\#\overline{E}(\F_q)$ has $|a_q| \leq 2\sqrt{q}$.

\begin{remark}
Here, we can see the advantage of working with the affine curve rather than the toric curve.  First and foremost, the computations are performed in a cohomology space of dimension $\deg L(f,T)=2$, as opposed to one of dimension $\deg L^*(f,T)=\Vol(\Delta)=6$.  This difference is accounted for by the number of points on $\overline{E}$ along the coordinate axes, as follows.  From the relation
\[ Z(\overline{E},qT)=L^*(wf,T)Z(\G_m^2,qT)=L^*(wf,T)\frac{(1-qT)^2}{(1-T)(1-q^2T)}, \]
after some cancellation we find that
\[ L^*(wf,T)=(1-T)P(qT)P_x(qT)P_y(qT) \]
where $Z(\overline{E} \cap \{x=0\},T)=P_x(T)/(1-T)$ is a polynomial of degree $1$ which is $1-T$ or $1+T$ depending on if $\overline{b} \in \F_q$ is a square or not; similarly, $P_y(T)$ is a polynomial of degree $2$ which depends on the factorization of $x^3+\overline{a}x+\overline{b}$ in $\F_q$. 
\end{remark}

We have the graded ring
\[ \Z_q[w\Delta] = \bigoplus_d \Z_q\la w^d x^i y^j : (i,j) \in d\Delta \ra \subseteq \Z_q[[w,x,y]] \]
and work in the Jacobian ring
\[ J = \frac{\Z_q[w\Delta]}{\la wf, wf_x, wf_y \ra}  \]
where 
\begin{align*}
f_x &= x\frac{\partial f}{\partial x} = 3x^3+ax \\
f_y &= -2y^2.
\end{align*}  

The affine Koszul complex associated to $f$ (see Section 5) is
\[
0 \to L_\Delta \xrightarrow{D} L_\Delta \oplus L_\Delta^{(x)} \oplus L_\Delta^{(y)} \xrightarrow{D} L_\Delta^{(x)} \oplus L_\Delta^{(y)} \oplus L_\Delta^{(xy)} \xrightarrow{D} L_\Delta^{(xy)} \to 0
\] 
where $L^{(m)}_\Delta \subseteq L_\Delta$ the subspace divisible by the monomial $m$, and $D=(D_w,D_x,D_y)$ where
\[ D_w = w\frac{\partial}{\partial w} + \pi wf, \quad D_x = x\frac{\partial}{\partial x} + \pi wf_x, \quad 
D_y = y\frac{\partial}{\partial y} + \pi w f_y. \]
We work in the cohomology space
\[ B = \frac{L_\Delta^{(xy)}}{D_w L_\Delta^{(xy)} + D_x L_\Delta^{(y)} + D_y L_\Delta^{(x)}}. \]

We then consider the free $\Z_q$-module
\[ J=\frac{\Z_q[w\Delta]^{(xy)}}{ \Z_q[w\Delta]^{(xy)} w f + \Z_q[w\Delta]^{(y)} wf_x + \Z_q[w\Delta]^{(x)} wf_y}. \]
In weight $1$ we have $J_1=\Z_q wxy$, since this is the only monomial divisible by $xy$ and it is nonzero in the quotient.  

We have
\begin{align*}
J_2 &= w^2\cdot\frac{\opspan(\{xy,x^2y,x^3y,x^4y,xy^2,x^2y^2,x^3y^2,xy^3\})}{\opspan(\{xyf, yf_x, xyf_x, y^2f_x, xf_y,x^2f_y,x^3f_y, xyf_y\})} \\
&= w^2xy\cdot\frac{\opspan(\{1,x,x^2,x^3\})}{\opspan(\{x^3+ax+b, 3x^2+a, 3x^3+ax\})} = \Z_q \cdot w^2xy.
\end{align*}
So the monomial basis $V$ we take is $wxy,w^2xy$.

Since $|a_q| \leq 2\sqrt{q}$, we recover $a_q$ uniquely as the integer $a$ such that $a_q \equiv a \pmod{p^N}$ and $|a|\leq 2\sqrt{q}$ with $p^N > 4\sqrt{q}$, so $N>\log_p 4+ (1/2)\log_p q$ or $N=O(\log q)$.  

For $v=wxy,w^2xy$, we expand as in (\ref{enumeq_5}).  We first take $v=wxy$.  Then we have 
\[ K=\{e: Me=-(1,1,1)^t \psmod{p}\} \]
where
\[ M=\begin{pmatrix} 1 & 1 & 1 & 1 \\ 3 & 1 & 0 & 0 \\ 0 & 0 & 2 & 0 \end{pmatrix}. \]
We compute then that
\[ K=\{-(1/4,1/4,1/2,0)+t(1,-3,0,2) : t \in \Z/p\Z\}. \]
Note that for $k=(k_1,k_2,k_3,k_4) \in K$ we have simply $k_3=(p-1)/2$.  Thus
\begin{equation} \label{enumeq_5_ec}
\begin{split}
\alpha(wxy) &= \sum_{k \in K} \sigma^{-1}(a)^{k_2}(-1)^{k_3}\sigma^{-1}(b)^{k_4} (-p)^{(|k|+1)/p}(\pi w)^{(|k|+1)/p} x^{(3k_1+k_2)/p} y \\
&\qquad\quad \cdot \left( \sum_{e \geq 0} (-p)^{|e|} \ell_{k+pe} (-1)^{e_3}a^{e_2}b^{e_4} (\pi w)^{|e|} x^{3e_1+e_2} y^{2e_3} \right). 
\end{split}
\end{equation}

We now reduce the series (\ref{enumeq_5_ec}).  Our algorithm to do this would proceed by degree; to make the resulting matrices digestible, we alter the method slightly.  We reduce first with respect to the monomial $y^2$, or ``vertically''.  We have
\[ D_y(w^d x^u y^{v-2})=(v-2) w^d x^u y^{v-2} + \pi w f_y(w^d x^u y^{v-2}) = 0 \in B, \]
and $wf_y = -2wy^2$ so 
\[ \pi w^d x^u y^v = \frac{1}{2}(v-2) w^{d-1} x^u y^{v-2} \in B. \]
By induction, if $v$ is odd we have
\[ (\pi w)^{\lfloor v/2 \rfloor} x^u y^v \equiv \frac{1}{2^{\lfloor v/2 \rfloor}} \left((v-2)(v-4)\cdots 3\cdot 1\right) x^u y = \lfloor v/2 \rfloor! (-1)^{\lfloor v/2 \rfloor}\binom{-1/2}{\lfloor v/2 \rfloor} x^u y. \]
Thus in the series (\ref{enumeq_5_ec}) we have
\begin{equation} \label{enumeq_5_ec_simplify_y}
\sum_{e \geq 0} \lambda_{k+pe} (-1)^{e_3} a^{e_2}b^{e_4} w^{|e|} x^{3e_1+e_2} y^{2e_3} =\sum_{e \geq 0} \frac{\lambda_{k+pe}}{\pi^{e_3}} e_3! \binom{-1/2}{e_3} a^{pe_2}b^{pe_4} w^{|e|-e_3} x^{3e_1+e_2}.
\end{equation}

We now reduce with respect to $x^3$, or ``horizontally'', then finally reducing with respect to the origin.  We compute the matrix $P_{x^3}$ such that
\[ w x^3 \begin{pmatrix} w^2x^2y \\ w^2x^3y \\ w^2x^4y \end{pmatrix} = P_{x^3} \begin{pmatrix} wf \\ wf_x \\ wf_y \end{pmatrix} \]
and having coefficients in $\Z_q \la w^2 x^2y, w^2x^3y,w^2x^4y \ra$.  Some linear algebra then yields
\[ P_{x^3}=\frac{w^2x^2y}{4a^3+27b^2}
\begin{pmatrix}
3ax(2ax-3b) & -2a^2x^2+3b(ax+3b) & -\frac{3}{2}ax(2ax-3b) \\
-ax(9bx+2a^2) & 3abx^2+(2a^3+9b^2)x+2a^2b & \frac{1}{2}ax(9bx+2a^2) \\
-a^2x(2ax-3b) & (2a^3+9b^2)x^2-ab(ax+3b) & \frac{1}{2}a^2x(2ax-3b) \\
\end{pmatrix}. \]
Now, if
\[ \xi = wf \eta_w + wf_x \eta_x + wf_y \eta_y \]
then
\[ -\pi\xi = \frac{\partial}{\partial w} \eta_w + \frac{\partial}{\partial x} \eta_x + \frac{\partial}{\partial y} \eta_y \in B. \]
This will allow us to write
\[ -\pi w^d x^u \begin{pmatrix} w^2x^2y \\ w^2x^3y \\ w^2x^4y \end{pmatrix} \equiv
D_{x^3}(d,u) w^{d-1} x^{u-3} \begin{pmatrix} w^2x^2y \\ w^2x^3y \\ w^2x^4y \end{pmatrix} \]
for a matrix $D_{x^3}(d,u)$ with coefficients in $\Z_q[d,u]$ which are linear in $d,u$.  

The rows of the matrix $D_{x^3}$ are obtained as follows.  For $i=1,2,3$, we have
\[ (wx^3)(w^2 x^{i+1} y) = p_{iw}f_w + p_{ix}f_x + p_{iy} f_y \]
where $(p_{iw},p_{ix},p_{iy})$ is the $i$th row of $P_{x^3}$.  Thus
\[ w^d x^u = (w^{d-1} x^{u-3})(wx^3)(w^2 x^{i+1} y) = (w^{d-1} x^{u-3})\left(p_{iw}f_w + p_{ix}f_x + p_{iy} f_y\right) \]
and hence
\begin{equation} 
\begin{split}
-\pi w^d x^u (w^2 x^{i+1} y) &\equiv w \frac{\partial}{\partial w}(w^{d-1} x^{u-3} p_{iw}) + 
x \frac{\partial}{\partial x}(w^{d-1} x^{u-3} p_{ix}) + y \frac{\partial}{\partial y}(w^{d-1} x^{u-3} p_{iy}) \notag \\
&= (d+1)w^{d-1}x^{u-3} p_{ix} + w^{d-1}\left( (u-3)x^{u-3}p_{ix}+x^{u-2}p_{ix}' \right) + w^{d-1}x^{u-3}p_{iy} \\
&= w^{d-1}x^{u-3} \left( (d+1)p_{iw} + (u-3)p_{ix} + xp_{ix}' + p_{iy} \right).
\end{split}
\end{equation}
From this we obtain that $(4a^3+27b^2)D_{x^3}(d,u)$ is equal to
\[ \begin{pmatrix}
-3ab^2(u-1) & \frac{1}{2}a^2b(6d-2u+3) & -2a^3d+(2a^3+9b^2)u + (a^3+9b^2) \\
2a^2b(u-1) & -2a^3d+(2a^3+9b^2)u-a^3 & -\frac{3}{2}ab(6d-2u+1) \\
9b^2(u-1) & -\frac{3}{2}ab(6d-2u-3) & a^2(6d-2u+1).
\end{pmatrix} \]

In a completely analogous manner, we obtain the matrix $D_1$ satisfying
\[ -\pi w^d x^u \begin{pmatrix} w^2x^2y \\ w^2x^3y \\ w^2x^4y \end{pmatrix} \equiv
D_{1}(d,u) w^{d-1} x^u \begin{pmatrix} w^2x^2y \\ w^2x^3y \\ w^2x^4y \end{pmatrix} \]
and indeed we have $b(4a^3+27b^2)D_1$ is equal to
\[ \begin{pmatrix}
(4a^3+27b^2)d -(4a^3+9b^2)u - \frac{1}{2}(12a^3 + 9b^2) & -3ab(6d-2u-3) & 2a^2(6d-2u-5) \\
4a^2b(u+2) & \frac{9}{2}b^2(6d-2u-3) & 3ab(6d-2u-5) \\
-6ab^2(u+2) & a^2b(6d-2u-3) & \frac{9}{2}b^2(6d-2u-5) 
\end{pmatrix}. \]
Note the additional need to invert $b$ for the matrix $D_1$.

To reduce the series (\ref{enumeq_5_ec})--(\ref{enumeq_5_ec_simplify_y}), we expand and rewrite it as
\begin{equation} \label{red_alphawxy}
\begin{split}
\alpha(wxy) &= \sum_{i \geq 0} w^i x^2y\left((c_{i0}+c_{i1}x+c_{i2}x^2) + wx^3(c_{i3}+c_{i4}x+c_{i5}x^2) + \dots \right) \\
& = \sum_{i \geq 0} w^i x^2y \left( \sum_{j \geq 0} x^{3j}(c_{i,3j+2}+c_{i,3j+3}x+c_{i,3j+4}x^2) \right) 
\end{split}
\end{equation}
with $c_{ij} \in \Z_q$.  We have
\[ \ord_p(c_{i,j'}) \geq (i+\lfloor j'/3 \rfloor) \frac{p-1}{p} \]
so up to precision $p^N$ we need only take $i<p/(p-1)N$ and
\[ j<J=\frac{p}{p-1}N-i. \]

Each sum in (\ref{red_alphawxy}) is reduced using the matrix $D_{x^3}$: if we abbreviate 
\[ c_i(j)=\begin{pmatrix}c_{i,3j+2} \\ c_{i,3j+3} \\ c_{i,3j+4} \end{pmatrix} \]
and write
\[ \widehat{D}_{x^3}(i,j)=D_{x^3}(i,0)D_{x^3}(i+1,3)\cdots D_{x^3}(j+i,3j) \]
we have
\[ w^i x^2y \sum_{j=0}^{J-1} c_{ij} w^{\lfloor j/3 \rfloor} x^j = w^ix^2y\sum_{j=0}^{J} \widehat{D}_{x^3}(i,j)c_i(j). \]

For what it is worth, one can check that this algorithm runs in time $O(p \log^6 q)$.

\subsection*{Fermat-like hypersurfaces}

In this subsection, we show how the method works in the simplest case where one has a Fermat-like affine hypersurface defined by
\[\fbar=\overline{a}_1x_1^{m_1}+\dots+\overline{a}_nx_n^{m_n}+\overline{b} \in \F_q[x_1,\dots,x_n]=\F_q[x] \]
where $m_i \in \Z_{>0}$ and $\overline{a}_1\cdots \overline{a}_n \overline{b} \neq 0$.  These were the varieties considered by Weil \cite{Weil} in his seminal article on zeta functions.  Koblitz \cite{Koblitz} has studied these and shown that the number of points is given by Jacobi sums, which can be expressed by Gauss sums; see Wan \cite{WanmodularGauss} for an explicit algorithm which uses this method.  In an early related work, Delsarte \cite{Delsarte} studied the number of zeros of a polynomial 
\[ \overline{g} = \sum_{j=1}^n \overline{b_j} x^{\nu(j)} + \overline{c} \in \F_q[x_1,\dots,x_n] \]
in $\F_q$ and its finite extensions, where $\nu(j) \in \Z_{\geq 0}^n$; he described an explicit formula for this number in terms of Jacobi sums.

The polynomial $\fbar$ is nondegenerate if and only if $p \nmid m_1\cdots m_n$, which we now assume.

Let $f=a_1x_1^{m_1}+\dots+a_nx_n^{m_n}+b \in \Z_q[x]$ be the Teichm\"uller lift of $f$ to $\Z_q[x]$.  We then have the polytope 
\[ \Delta=\Delta(f)=\Delta(\{(m_1,0,\dots,0),\dots,(0,\dots,0,m_n),(0,\dots,0)\}) \]
with normalized volume $\Vol(\Delta)=n!\vol(\Delta)=m_1\cdots m_n$.  

Here, the affine complex gives the cohomology space 
\[ B=\frac{L^{(x_1\cdots x_n)}}{D_w L^{(x_1\cdots x_n)} + D_1 L^{(x_2 \cdots x_n)} + \dots + D_n L^{(x_1 \cdots x_{n-1})}} \]
where 
\[ D_w=w\frac{\partial}{\partial w}+\pi f_w, \quad D_i=x_i \frac{\partial}{\partial x_i} + \pi f_{x_i} \]
for $i=1,\dots,n$, and
\[ f_w=wf, \quad f_{x_i}=(wa_im_i) x^{m_i} \]
also for $i=1,\dots,n$.

A basis $V$ for the space $B$ is computed as follows.  In weight $1$, we have $V_1$ consisting of lattice points in $\Delta$ not on a coordinate face, 
\[ V_1=\{wx^{\mu} : (1/m)\mu = \textstyle{\sum}_i \mu_i/m_i \leq 1,\ \mu>0\}. \]
In a similar way, in degree $d$, we obtain
\[ V_d=\{w^d x^\mu : i-1<(1/m)\mu \leq i,\ \mu>0,\ \mu<m \}. \]
In particular, we see visibly that $V_d=\{0\}$ for $d \geq n+1$.

Now we consider each series expansion (\ref{enumeq_6}).  
We have
\[ K=\{e:Ue=\mu \psmod{p}\} \]
where
\[ U=\begin{pmatrix}
1 & 1 & \dots & 1  \\
0 & m_1 & \dots & 0 \\
\vdots & \vdots & \ddots & \vdots \\
0 & 0 & \dots & m_n 
\end{pmatrix} \]
which is an invertible $(n+1)\times (n+1)$ matrix, so $K$ consists of the single element
\[ k=M^{-1}(-\mu)=
\begin{pmatrix}
-1 & m_1^{-1} & m_2^{-1} & \dots & m_n^{-1} \\
0 & -m_1^{-1} & 0 & \dots & 0 \\
0 & 0 & -m_2^{-1} & \dots & 0 \\
\vdots & \vdots & \vdots & \ddots & \vdots \\
0 & 0 & 0 & \dots & -m_n^{-1}
\end{pmatrix} \mu. \]
For simplicity of notation, we consider the indices of the columns to be $0,\dots,n$.

The reduction steps are similarly simple.  Since $f_i=a_im_i w x_i^{m_i}$, we have
\[ (\pi w)^d x^\mu \equiv -\frac{\mu_i-m_i}{a_i} (\pi w)^{d-1} x^\mu x_i^{-m_i} \in B \]
whenever $\mu_i \geq m_i$.  Similarly, from the equality
\[ w^d = \frac{1}{b}w^{d-1}f_w - \frac{1}{b(m_1\cdots m_n)} \sum_{i=1}^n w^{d-1} f_i \]
which implies
\[ (\pi w)^d \equiv -\frac{d-1}{b} (\pi w)^{d-1} \in B. \]
Inductively packing these up, in the terms of the expansion
\begin{align*} 
\alpha((\pi w)^d x^{\mu}) &= \sigma^{-1}(a^k) (-p)^{(|k|+d)/p} (\pi w)^{(|k|+d)/p} x^{(k\nu+\mu)/p} \left( \sum_{e=(e_0,\dots,e_n) \geq 0} (-p)^{|e|} \ell_{k+pe} a^e (\pi w)^{|e|} x^{e\nu}\right) \\
&= \sigma^{-1}(a^k) (-p)^{(|k|+d)/p} \sum_{e \geq 0} (-p)^{|e|}\ell_{k+pe} a^e (\pi w)^{|e|+(|k|+d)/p} x^{e\nu+(k\nu+\mu)/p}
\end{align*}
and letting $\kappa = (k\nu+\mu)/p=\bigl((k_i m_i + \mu_i )/p\bigr)_{i=1,\dots,n}$ we have
\begin{align*}
(\pi w)^{|e|+(|k|+d)/p} x^{e\nu + (k\nu+\mu)/p} \equiv &(-1)^{e_0} (e_0 + (|k|+d)/p; 1) \\
&\qquad \cdot \prod_{i=1}^{n} (-a_i^{-1})^{e_i} (e_im_i + \kappa_i;m_i)_{e_i} (\pi w)^{(|k|+d)/p} x^\kappa 
\end{align*}
where
\[ (z;q)_r=(z-q)(z-2q) \cdots (z-rq). \]
If we let $m_0=1$ and $\kappa_0=(|k|+d)/p$, then we can abbreviate
\[ (em+\kappa;m)_e = \prod_{i=0}^{n} (e_i m_i + \kappa_i;m_i)_{e_i} \]
and substituting back into the sum, we obtain
\[ \alpha((\pi w)^d x^{\mu}) = (\pi w)^{(|k|+d)/p} x^\kappa \left(\sigma^{-1}(a^k) (-p)^{(|k|+d)/p} \sum_{e \geq 0} p^{|e|}\ell_{k+pe} (em+\kappa;m)_e \right). \]
In this way, we have written down the reduced value in one stroke.  

\subsection*{Gabber hypersurfaces}

There are few works (if any) in the existing literature which give the computation of the zeta function of a projective hypersurface defined over $\F_q$ of degree $d$ with $p \mid d$.  Here we study the zeta function of such hypersurfaces for a particular family going back (according to oral communication from Nicholas Katz) to Ofer Gabber (see for example \cite[11.4.6]{KatzSarnak}.

Let $d$ be a positive integer with $p \mid d$ and let
\begin{equation} \label{Gabber}
\overline{f}(x)=a_1x_1^d + \sum_{i=2}^n a_i x_{i-1} x_i^{d-1} \in \F_q[x_1,\dots,x_n] 
\end{equation}
with $a_1a_2 \cdots a_n \neq 0$.  Let $\overline{Z}_0$ be the projective hypersurface in $\PP^{n-1}_{\F_q}$ defined by $\overline{f}=0$.  The hypersurface $\overline{Z}_0$ is easily seen to be nonsingular.  Note that its defining equation (\ref{Gabber}) is fewnomial.  

We consider somewhat more general hypersurfaces as follows.  Let $c=(c_1,\dots,c_n)$ be an $n$-tuple of positive integers such that $|c|=\sum_{i=1}^n c_i=d$.  Suppose further that $c$ is an interior point in $\Delta(\fbar)$.  We consider here the family of hypersurfaces $\overline{Z}_\lambda$ defined by
\[ \overline{f}_\lambda(x) = \fbar(x) + \lambda x^c. \]
Since $\overline{Z}_0$ is nonsingular and the condition of nonsingularity is open, there is a closed subset $W \subseteq \A^1$, defined by the vanishing of a polynomial in $\F_q[\lambda]$, such that $\overline{Z}_\lambda$ is singular if and only if $\lambda \in W$.  

If $\lambda \not\in W$, then we recall 
\[ Z(\overline{Z}_\lambda, T) = \frac{ P_\lambda(T)^{(-1)^{n-1}} }{(1-T)(1-qT) \cdots (1-q^{n-2}T)} \]
with the degree of $P_\lambda$ equal to $d^{-1}\bigl( (d-1)^n + (-1)^n (d-1) \bigr)$.  

\begin{example}
An example of this situation is provided by the family of elliptic curves
\[ \overline{Z}_\lambda: x_1^3 + x_1x_2^2 + x_2x_3^2 + \lambda x_1x_2x_3 \]
in characteristic $3$.  Here the singular locus is given by $Q(\lambda)=\lambda^4-1=0$, the fourth roots of $1$.  We note that the $j$-invariant of $\overline{Z}_\lambda$ is equal to $j(\overline{Z}_\lambda) = \lambda^{12}/(\lambda^4-1)$.
\end{example}

It will be useful to consider the exponential sums (and their associated $L$-functions) for $w\fbar_\lambda(x)$ as $(w,(x_1,\dots,x_{n-1}),x_n)$ runs over various spaces which are products of tori with affine spaces.  For example, we recall from the sections on affine and projective hypersurfaces in Section 5 that
\begin{equation} \label{PqPq2}
L(w\fbar_\lambda,\G_m \times \A^{n-1} \times \A^1,T)^{(-1)^{n+1}} = \frac{P_\lambda(qT)}{P_\lambda(q^2T)}.
\end{equation}
It is our intention to show that (\ref{PqPq2}) has the form 
\[ \frac{P_\lambda(qT)}{P_\lambda(q^2T)} = \frac{\prod_{r=0}^{n-1} R_r(T)}{\prod_{r=0}^{n-1} R_r(qT)} \]
where $R_r(T)$ is a polynomial or reciprocal polynomial in $\Z[T]$ for $r=0,\dots,n-1$ which we give explicitly.  But then $\prod_{r=0}^{n-1} R_r(T) = P_\lambda(qT)$ so that this calculation of $R_r(T)$ for $0 \leq r \leq n-1$ in fact gives $Z(\overline{Z}_\lambda,T)$.

Since $\fbar_\lambda(x)$ is homogeneous, we have $\dim \Delta(\fbar_\lambda)=n-1$.  For $\lambda \not\in W$, the nonsingularity of $\overline{Z}_\lambda$ implies that $w\fbar_\lambda$ is nondegenerate with respect to its maximal face (not containing $0$).  It is easy to see it is also nondgenerate with respect to its other (lower-dimensional) faces not containing $0$.  

We cannot apply directly our results from the section on projective hypersurfaces in Section 5 because $p$ divides the degree $d$ and moreover $w\fbar_\lambda$ is not convenient with respect to the variables $x_1,\dots,x_n$.  It is, however, convenient with respect to $x_n$.  We proceed by partitioning $\G_m \times \A^{n-1} \times \A^1$ and particularly its middle factor as follows:
\begin{equation} \label{decomposeAn}
\A^{n-1} = \G_m^{n-1} \cup \bigcup_{r=1}^{n-1} U_r
\end{equation}
where $U_r = \G_m^{r-1} \times \{0\} \times \A^{n-1-r}$.  For notational convenience, we will sometimes write $U_0=\G_m^{n-1}$.  Note that the terms $U_r$ for $r=0,\dots,n-1$ are pairwise disjoint, so that if we set $U_r'=\G_m \times U_r \times \A^1$ then 
\[ L(w\fbar_\lambda,\G_m \times \A^{n-1} \times \A^1,T) = \prod_{i=0}^{n-1} L(w\fbar_\lambda,U_r',T). \]

By work of Adolphson and Sperber \cite{ASExpsums}, the complex $\Omega^{\omegadot}(w\fbar_\lambda,U_0')$ for $L(w\fbar_\lambda,U_0',T)$ is acyclic except in dimensions $n$ and $n+1$ and 
\[ L(w\fbar_\lambda,U_0',T)^{(-1)^{n+1}} = \frac{R_0(T)}{R_0(qT)} \]
where
\[ R_0(T) = \det \bigl(1-\Frob T \mid H^{n+1}(\Omega^{\omegadot}(w\fbar_\lambda,U_0') \bigr). \]
We see easily, from the relation of $R_0(T)$ and the zeta function for the variety defined by the vanishing of $\fbar_\lambda$ in $\G_m^{n-1} \times \A^1$, that $R_0(T) \in \Z[T]$.  The calculation of $R_0(T)$ follows easily from the suppression of the $w$ terms and the isomorphism
\[ H^{n+1}(\Omega^{\omegadot}(w\fbar_\lambda,U_0')) \cong \frac{(L_\Delta')^{\{x_n\}}}{\sum_{i=1}^{n-1} D_i (L_\Delta')^{\{x_n\}} + D_n L_\Delta'}. \]

Consider now for $1 \leq r \leq n-1$ the $L$-function associated with the exponential sum $S(w\fbar_\lambda,U_r')$.  Write 
\[ \fbar_\lambda^{(r)}(x) = a_1 x_1^d + \sum_{i=1}^{r-1} a_i x_{i-1} x_i^{d-1} \]
and
\[ \overline{g}_\lambda^{(r)}(x) = \sum_{i=r+2}^n a_i x_{i-1} x_i^{d-1}. \]
Note that substituting $x_r=0$ in $\fbar_\lambda(x)$ gives
\begin{equation} \label{subsxr0}
\left. \fbar_\lambda(x) \right|_{x_r=0} = \fbar_{\lambda}^{(r)}(x) + \overline{g}_{\lambda}^{(r)}(x).
\end{equation}
Note also that despite the notation, none of the polynomials in (\ref{subsxr0}) (for $r \geq 1$) depend on $\lambda$.  Thus 
\[ S(w\fbar_\lambda,U_r') = \sum_{(w,x_1,\dots,x_{r-1})} \Theta(w\fbar_{\lambda}^{(r)}(x_1,\dots,x_{r-1}))
\left( \sum_{(x_{r+1},\dots,x_n) \in \A^{n-r}} \Theta(w\overline{g}_\lambda^{(r)}(x_{r+1},\dots,x_n)) \right). \]
For any $w \in \G_m$, the inner sum is easily seen to be $q^{n-r-1}$, i.e.
\[ S(w\fbar_\lambda,U_r') = q^{n-r-1} S(w\fbar_\lambda^{(r)}, \G_m^r) \]
and
\[ L(w\fbar_\lambda,U_r',T)=L(w\fbar_\lambda^{(r)}, \G_m^r, q^{n-r-1}T). \]

Consider the complex $\Omega^{\omegadot}(w\fbar_\lambda^{(r)},\G_m^r)$ for the $L$-function $L(w\fbar_\lambda^{(r)}, \G_m^r, T)$.  It is acyclic $H^i(\Omega^{\omegadot})=0$ except for $i=r-1$ and $i=r$ and if we let
\[ H^{(r)}(T) = \det\bigl(1-\Frob T \mid H^r(\Omega^{\omegadot}(w\fbar_\lambda^{(r)},\G_m^r))\bigr) \]
then
\[ L(w\fbar_\lambda^{(r)}, \G_m^r, T)^{(-1)^{r+1}} = \frac{H^{(r)}(T)}{H^{(r)}(qT)}. \]
Thus if we let
\[ R_r(T)=\det\bigl(1- q^{n-r-1}\Frob T \mid H^r(\Omega^{\omegadot}(w\fbar_\lambda^{(r)},\G_m^r))\bigr)^{(-1)^{n-r}}, \]
then as with $R_0(T)$, we have $R_r(T) \in \Z[T]$.  

We can calculate $R_r(T)$ using
\[ H^r(\Omega^{\omegadot}(w\fbar_\lambda^{(r)},\G_m^r))\bigr) \cong \frac{L_{\Delta^{(r)}}}{ \sum_{i=1}^{r-1} D_i L_{\Delta^{(r)}}} \]
where we have suppresed the factor of $w$.  Note that in the case $r=1$, we have 
\[ L(w\fbar_\lambda,U_1',T) = \frac{1-q^{n-2}T}{1-q^{n-1}T} \]
so that $R_1(T)=(1-q^{n-2}T)^{(-1)^{n+1}}$.  Similarly, if $r=2$, then
\[ L(w\fbar_\lambda,U_2',T) = \frac{1-q^{n-2}T}{1-q^{n-3}T} \]
so $R_2(T)=(1-q^{n-3} T)^{(-1)^{n}}$.  

\begin{acknowledgements}\label{ackref}
This work was initiated during the Thematic Year on Applications of Algebraic Geometry at the Institute for Mathematics and its Applications (IMA) in 2006--2007, and the authors would like to thank the IMA for their hospitality.   We would like to thank David Harvey, Alan Lauder, Kiran Kedlaya, and Daqing Wan for helpful discussions; Wouter Castryck, Frank Sottile, and the anonymous referee for some corrections; and Jesus de Loera, Benjamin Nill, and Bernd Sturmfels for answering some questions about algorithms for polytopes.
\end{acknowledgements}

\affiliationone{
Steven Sperber \\
School of Mathematics \\ University of Minnesota \\ 206 Church Street SE \\ Minneapolis, MN 55455 \\ USA \\
\email{sperber@math.umn.edu}}
\affiliationtwo{
John Voight \\
Department of Mathematics and Statistics \\ University of Vermont \\ 16 Colchester Ave \\ Burlington, VT 05401 \\ USA \\
\email{jvoight@gmail.com}}
\end{document}